\documentclass[preprint,12pt]{elsarticle}

\usepackage{hyperref}
\hypersetup{
    colorlinks,
    citecolor=black,
    filecolor=black,
    linkcolor=black,
    urlcolor=black
}

\usepackage{amssymb}
\usepackage{amsmath}
\usepackage{mathtools}
\usepackage{amsfonts}
\usepackage{amsthm}
\usepackage{float}
\usepackage{enumitem}
\usepackage{times}
\usepackage{tikz-cd}
\usetikzlibrary{shapes.geometric}
\usepackage{mathabx}
\usepackage{stmaryrd}
\usepackage{multicol}
\usetikzlibrary{arrows}
\tikzset{global scale/.style={
    scale=#1,
    every node/.style={scale=#1}
  }
}
\usepackage{comment}

\RequirePackage[normalem]{ulem}

\floatstyle{boxed}
\newtheorem{theorem}{Theorem}[section]
\newtheorem{proposition}[theorem]{Proposition}
\newtheorem{lemma}[theorem]{Lemma}
\newtheorem{corollary}[theorem]{Corollary}

\newtheorem*{claim*}{Claim}
\theoremstyle{definition}

\newtheorem{remark}[theorem]{Remark}

\newcommand{\va}{\mathsf} 
\newcommand{\m}{\mathbf} 

\newcommand{\jn}{\mathbin{\vee}}
\newcommand{\mt}{\mathbin{\wedge}}
\newcommand{\ZZ}{\mathbb{Z}}
\newcommand{\NN}{\mathbb{N}}
\newcommand{\vr}{\mathbb{V}}

\newcommand{\Grp}[1]{{#1_{g}}}
\newcommand{\grp}[1]{{#1_{g}}}
\newcommand{\cn}[1]{\boldsymbol{#1}}
\newcommand{\g}{\sigma}
\newcommand{\h}{\gamma}
\newcommand{\slt}{\delta}
\DeclareMathOperator{\lcm}{lcm}
\DeclareMathOperator{\supp}{supp}
\newcommand{\loc}[1]{{#1}_{\mathrm{loc}}}
\newcommand{\glob}[1]{{#1}_{\mathrm{glo}}}

\DeclareMathOperator{\per}{per}
\newcommand{\comp}[1]{{#1}_{\mathrm{cp}}}

\newcommand{\f}{\varphi}
\newcommand{\ps}{\psi}
\newcommand{\cc}{\mathbf{c}}

\journal{-}

\begin{document}

\begin{frontmatter}

\title{Proper varieties of distributive  \texorpdfstring{$\ell$}{ℓ}-pregroups are periodic}
\author[inst1]{Nikolaos Galatos}
\ead{ngalatos@du.edu}

\affiliation[inst1]{organization= {Department of Mathematics, University of Denver},
            addressline={2390 S. York St.}, 
            city={Denver},
            postcode={80208}, 
            state={CO},
            country={USA}}

\author[inst2]{Simon Santschi}
\ead{santschi@cs.cas.cz}
\affiliation[inst2]{organization= {Institute of Computer Science, Czech Academy of Sciences},
            addressline={Pod Vodárenskou věží 271/2}, 
            city={Praha 8},
            postcode={182 00}, 
            country={Czechia}} 

\begin{abstract}
We show that every proper 
variety of distributive $\ell$-pregroups is periodic; as a result, the subvariety lattice of  the variety of distributive $\ell$-pregroups decomposes into a disjoint union of bounded intervals each consisting of all properly $n$-periodic varieties, for different $n \in \ZZ^+$.  
On the way, we study non-periodic (one-generated) subalgebras of $\m F(\ZZ)$ and show that they either contain a non-periodic idempotent or they contain idempotents of arbitrary big periodicity. 
The same results holds for non-periodic one-generated subalgebras of $\m{F}(\m{J} \overrightarrow{\times} \ZZ)$, where $\m J$ is a chain, generated by a composite element and the resulting idempotents can be taken to be local in the associated wreath product. 
\end{abstract}
\end{frontmatter}

\section{Introduction}\label{s: intro}
A \emph{lattice-ordered pregroup} (\emph{$\ell$-pregroup}) is an algebra
$(A,\mt,\jn,\cdot,1,{}^{\ell} ,{}^{r})$,
such that $(A,\mt,\jn)$ is a lattice, $(A,\cdot,1)$ is a monoid, multiplication preserves the lattice order, and for all $a\in A$,
\[
a^{\ell} a\leq1\leq aa^{\ell} \text{ and }aa^{r}\leq1\leq a^{r}a.
\]

The $\ell$-pregroups that satisfy $x^\ell\approx x^r$ are the well-known and extensively studied $\ell$-groups (see \cite{AF, KM, Da}), and many results about various classes of $\ell$-pregroups extend existing results about $\ell$-groups. 
Versions of $\ell$-pregroups where the order is not assumed to be a lattice order are known as \emph{pregroups} and are of importance in mathematical linguistics, both theoretical and applied; (see \cite{La, Bu, Ba}). 

Moreover, $\ell$-pregroups are precisely the involutive residuated lattices for which multiplication is its own  De Morgan dual. 
As residuated lattices include structures such as Boolean algebras, Heyting algebras, relation algebras, and ideals lattices of rings, among others, the structure of $\ell$-pregroups connects to many other parts of ordered algebra. 
Also, as residuated lattices form algebraic semantics for substructural logics,  varieties of $\ell$-pregroups correspond to various such logics; see \cite{GJKO}. 

An $\ell$-pregroup is called \emph{distributive} if its underlying lattice is distributive; the variety of all distributive $\ell$-pregroups is denoted by $\mathsf{DLP}$. 
Examples of distributive $\ell$-pregroups include algebras of the form $\m F(\m \Omega)$, where $\m \Omega$ is a chain; the underlying set $F(\m \Omega)$ consists of all (order-preserving) maps on $\m \Omega$ that have residuals and dual residuals of all orders (see next section), multiplication is functional composition, $1$ is the identity function on $\Omega$, the order, join and meet are pointwise, and $f^r$ and $f^\ell$ are the the residual and the dual residual of the function $f$, respectively. 
 
 We call an element $a$ in an $\ell$-pregroup  \emph{$n$-periodic}, for  $n \in \NN$, if it satisfies  $a^{\ell^n}=a^{r^n}$; the $n$-periodic elements of an $\ell$-pregroup always form a subalgebra of the $\ell$-pregroup. 
 If all elements are $n$-periodic, the $\ell$-pregroup itself will be called \emph{$n$-periodic} and the variety of all $n$-periodic $\ell$-pregroups is denoted by $\mathsf{LP}_n$.
 
 It is well-known and easy to prove that all $\ell$-groups are distributive, i.e., $\mathsf{LP}_1 \subseteq \mathsf{DLP}$; in \cite{GJ} it is actually shown that $\mathsf{LP}_n \subseteq \mathsf{DLP}$, for all $n \in \ZZ^+$, so periodicity implies distributivity. 
 Also, in \cite{GG1} it is proved that $\bigvee_{n \in \ZZ^+}\mathsf{LP}_n = \mathsf{DLP}$. 
 In this work, we prove that actually every proper subvariety of $\mathsf{DLP}$ is already contained in $\mathsf{LP}_n$ for some $n$, i.e., that every proper variety of distributive $\ell$-pregroups is periodic.

 While working with distributive $\ell$-pregroups, we will be making use of their representation. 
 In particular, in \cite{GJ}, in analogy to Cayley's representation for groups and Holland's representation for $\ell$-groups, it is shown that every distributive $\ell$-pregroup can be embedded in  $\m F(\m \Omega)$, for some chain $\m \Omega$, and in \cite{GG1} it is proved that $\m \Omega$ can be taken to be discrete, i.e., of the form $\m{J} \overrightarrow{\times} \ZZ$, where $\m J$ is a chain.
 This specializes, as shown in \cite{GG2}, to obtain  that every $n$-periodic $\ell$-pregroup can be embedded in $\m F_n(\m{J} \overrightarrow{\times} \ZZ)$, for some chain $\m J$; here  $\m{F}_n(\m{\Omega})$ denotes the subalgebra of the $n$-periodic elements of   $\m{F}(\m{\Omega})$.
 This is important as in \cite{GG2} the notion of a wreath product of an $\ell$-group and an $\ell$-pregroup is defined (extending the usual notion of wreath products for groups, monoids, and $\ell$-groups) and it is shown that $\m F_n(\m{J}\overrightarrow{\times} \ZZ)$ is isomorphic to the wreath product $\m{Aut}(\m{J}) \wr \m{F}_n(\ZZ)$, where $\m{Aut}(\m{J})$ is the $\ell$-group of all order-preserving automorphisms on $\m J$. 
 This wreath product decomposition does not extend to the setting of $\m F(\m{J}\overrightarrow{\times} \ZZ)$, but in Section~\ref{s: composite} we identify a suitable subalgebra of $\m F(\m{J}\overrightarrow{\times} \ZZ)$ (of \emph{composite elements}) that is isomorphic to the wreath product $\m{Aut}(\m{J}) \wr \m{F}(\ZZ)$ and plays an important role in our study.

To prove our main result, that every proper variety of distributive $\ell$-pregroups is periodic, we consider an arbitrary non-periodic distributive $\ell$-pregroup $\m L$ and show that it generates all of $\mathsf{DLP}$. Toward that goal, in Section~\ref{s: dichotomy} we prove a dichotomy theorem: such an $\m L$ either contains a non-periodic idempotent or it contains idempotents of arbitrary big periodicity. The second case can already be handled using results in Section~\ref{s: prelims}, as Corollary~\ref{cor:per-gen-DLP} states that in this case $\m L$ generates all of $\mathsf{DLP}$. The first case requires more work and the new notion of \emph{simulation}.
 
The decomposition results of \cite{GG1, GG2} mentioned above, are used there to establish  generation results: $\mathsf{DLP}=\vr(\m F(\ZZ))=\vr(\m F_\mathrm{fs}(\ZZ))=\bigvee_{n \in \ZZ^+}\vr(\m F_n(\ZZ))$; here $\m F_\mathrm{fs}(\ZZ)$ denotes the subalgebra of  $\m F(\ZZ)$ of functions with finite support. 
In Section~\ref{s: simulation} we define the notion of an $\ell$-pregroup \emph{simulating} an other $\ell$-pregroup and, using the above generation results, we show that if an $\ell$-pregroup simulates $\m F_\mathrm{fs}(\ZZ)$, or it simulates $\m F_n(\ZZ)$ for infinitely-many $n\in \ZZ^+$, then it generates the variety $\mathsf{DLP}$.

In Section~\ref{s: main} we define a term that improves  a positive idempotent to one where its maximal height (difference between input and output) is controlled. This is then used to prove that  every non-periodic positive idempotent  of $ \m{F}(\ZZ)$ generates a subalgebra that in turn generates all of $\mathsf{DLP}$. These two facts are the last missing ingredients to handle the remaining case of the dichotomy theorem, thus yielding the main result that all proper subvarieties of $\mathsf{DLP}$ are periodic.

Combining our main theorem with results in \cite{GS}, it follows that  the subvariety lattice of $\mathsf{DLP}$ decomposes into a disjoint union of bounded intervals each consisting of all properly $n$-periodic varieties, for different $n \in \ZZ^+$. 
As a result, the study of the subvariety lattice of $\mathsf{DLP}$ reduces to the study of these $n$-periodic intervals and the relations between them.

In \cite{GG1, GG2}, the generation results  mentioned above are accompanied with a proof of the decidability of the equational theory of $\mathsf{DLP}$, obtained by undertaking a complex analysis of failures of equations in the generating algebras.
We mention that the same generation results, combined with our main result that all proper subvarieties of $\mathsf{DLP}$ are periodic, can be used to provide an alternative, and much simpler and shorter, proof that $\mathsf{DLP}$ is decidable.\footnote{For the sake of transparency, we note that the generation results and the decidability results in \cite{GG1, GG2} are proved simultaneously, in a single overall argument. Also, the relative simplicity of our new proof of decidability is counterbalanced by the fact that the much more involved proof in \cite{GG1, GG2} is accompanied by a complexity bound on the decision algorithm.}

We conclude the paper with Section~\ref{s: app}, where we use the main theorem to obtain very simple axiomatizations of varieties of the form   $\bigvee_{k= 1}^n \mathsf{LP}_k$, for $n\in \ZZ^+$.

\section{Preliminaries}\label{s: prelims}
 Recall that for functions $f,g$ on a poset $\m \Omega$, we say that $g$ is the residual of $f$ (and write $g=f^r$) and that $f$ is the dual residual of $g$ (and write 
 $f=g^\ell$) if for all $x,y \in \Omega$, we have $f(x) \leq y \Leftrightarrow x \leq g(y)$. 
 In case that the residual and dual residual of $f$ exist, they are given by the formulas
\[
f^{\ell}(x)=\min \{ y\in \Omega : x\leq f (y) \}
\quad \text{ and } \quad 
f^{r}(x)=\max \{ y\in \Omega : f(y)\leq x \}.
\]
Also, for all $n \in \NN$, the \emph{$n$th residual} of $f$ or \emph{residual of order $n$}, if it exists, is denoted by $f^{\ell^n}$ and is defined inductively by $f^{\ell^0}=f$, $f^{\ell^{n+1}}=(f^{\ell^n})^\ell$; likewise we define the \emph{$n$th dual residual} of $f$ or \emph{dual residual of order  $n$} and we denote it by $f^{r^n}$.

 Given a chain $\m \Omega$,  $F(\m \Omega)$ denotes the set of all (order-preserving) maps on $\m \Omega$ that have residuals and dual residuals of all orders. 
 As mentioned in the introduction, this set supports an algebra $\m F(\m \Omega)$, where multiplication is functional composition, $1$ is the identity function on $\Omega$, the order, join and meet are pointwise, and $f^r$ and $f^\ell$ are the residual and the dual residual of the function $f$, respectively. 
 We mention that, in particular, the functions in $\m F(\ZZ)$ end up being the functions on $\ZZ$ that are order-preserving and finite-to-one, i.e., the preimage of every point is a finite (possibly empty) interval.
 If the chain $\m \Omega$ is discrete, every element $x \in \Omega$ has an upper cover, denoted by $x+1$, and a lower cover, denoted by $x-1$; we also define $x+k$ inductively for  $k \in \ZZ$. 
 In this case, the functions $f \in \m F (\m \Omega)$ that satisfy $f^{\ell^n}=f^{r^n}$ are exactly the ones that repeat themselves periodically with period $n \in \NN$, i.e., $f(x+n)=f(x)+n$ for all $x \in \Omega$, thus justifying the term $n$-periodic for arbitrary elements of $\ell$-pregroups.
 Finally, we set $f^{[n]}:=f^{\ell^{2n}}$ and $f^{[-n]}:=f^{r^{2n}}$, for $n\in \NN$.
If the chain $\m \Omega$ is discrete, then $f^{[n]}$ can be calculated in a simple way, as shown in Lemma~2.3 of \cite{GG2}: 
for $f \in \mathrm{F}(\m \Omega)$ and $n\in \ZZ$, $f^{[n]}(x) = f(x-n) + n$. 
We also mention that the maps $x \mapsto x^{[k]}$ are automorphisms on every $\ell$-pregroup; see Lemma~2.4 of \cite{GS}.

As we mentioned,  $\m{F}_n(\m{\Omega})$ denotes the subalgebra of  $\m{F}(\m{\Omega})$ consisting of all $n$-periodic functions. 
Also, if $a$ is an element of  an $\ell$-pregroup, the \emph{periodicity} $\per(a)$ of $a$ is defined to be the minimum  $k\geq 1$ such that $a$ is $k$-periodic; we write $\per(a)=\infty$ if no such minimum exists. 

Recall that $\m{F}_\mathrm{fs}(\ZZ)$ denotes the subalgebra of $\m{F}(\ZZ)$ consisting of finite support members, i.e., the elements $f\in \mathrm{F}(\ZZ)$ where $\{x \in \ZZ : f(x) \neq x\}$ is finite.

\begin{proposition}[{\cite[Corollary~4.2]{GG1}}]\label{p:DLP-gen-fs}
  $\vr(\m{F}_\mathrm{fs}(\ZZ)) = \va{DLP}$.
\end{proposition}
In the following, if $X$ is a subset of an algebra $\m A$, we denote by $\langle X \rangle$ the subalgebra of $\m A$ generated by $X$; also, 
we write $\langle x \rangle$ for $\langle \{x\} \rangle$. 

We define $\cc \in \mathrm{F}_\mathrm{fs}(\ZZ)$ by 
\[
\cc(k) = \begin{cases}
    1 & \text{if } k=0, \\
    k & \text{otherwise.}
\end{cases}
\]
\begin{lemma}[{\cite[Theorem~6.8]{Gallardo}}]\label{l:Ffs-gen}
     $\m{F}_\mathrm{fs}(\ZZ)$ is  generated by $\cc$ (using only multiplication, $^\ell$, $^r$, $1$). 
\end{lemma}

\begin{proof}
    It suffices to show that $\cc$ generates each positive element of $\m{F}_\mathrm{fs}(\ZZ)$ (i.e., elements $h\geq 1$), since for each $f\in \mathrm{F}_\mathrm{fs}(\ZZ)$ we have $f = (f \jn 1)(f \mt 1)=(f \jn 1)(f^\ell \jn 1)^r$. 
    So, let $f$ be a positive element in $ \mathrm{F}_\mathrm{fs}(\ZZ)$; since the support of $f$ is finite, it is contained in an interval $[s,t]$, where $s,t \in \ZZ$. 
    So, the support of $g:=f^{[-s]}$ is contained in the interval $[0,t-s]$ and $g$ is also positive.

    We denote by $\{k_1 <k_2< \ldots< k_{r}\}$ the set of all $k$ 
    that are in the image of $g$ and  $g^{-1}[k]\neq \{k\}$.
    For every such $k$,  we set $l_k=\min g^{-1}[k]+1$ and note that, because $g$ is positive, we have $\min g^{-1}[k]<k$, so $l_k\leq k$; we define the product $g_k:=c_{k}\cdots c_{l_k}$, where $c_i=\cc^{[i-1]}$, i.e., $c_i(i-1)=i$ and $c_i$ is the identity on all other inputs; see Figure~\ref{fig:c-generatese-fs}. 
    We will verify that $g=g_{k_1}g_{k_2}\cdots g_{k_r}$, where the empty product is the identity function; it is enough to check this equality on $[0,t-s]$, 
    as outside the interval both functions behave like the identity. 
    Indeed, if $g(x)=y$ for some $x \in [0,t-s]$,  then in case $g^{-1}[y]= \{y\}$, we get $y=x$ and $g_{k_1}g_{k_2}\cdots g_{k_r}(x)=x=y$; 
    in case $g^{-1}[y] \neq \{y\}$, then $y=k_i$, for some $i$ and $x \in g^{-1}[k_i]$. 
    Also $g_{k_i}(x)=k_i$, while $g_{k_j}(x)=x$, for $j \neq i$. 
    So, $g_{k_1}g_{k_2}\cdots g_{k_r}(x)=g_{k_1}g_{k_2}\cdots g_{k_i}(x)=g_{k_1}g_{k_2}\cdots g_{k_{i-1}}(k_i)=k_i=y$, because $ g_{k_j}(k_i)=k_i$ for $j<i$. 
    Therefore, $g \in \langle \cc \rangle$, hence $f=g^{[s]} \in \langle \cc \rangle$.
\end{proof}

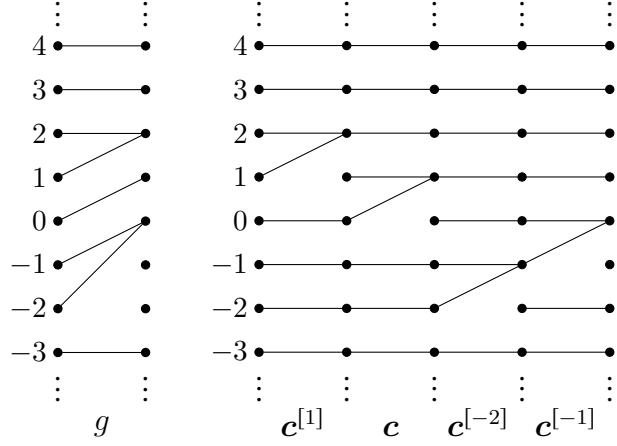
\begin{figure}
    \centering
    \begin{tikzpicture}[scale=0.58]
\node[fill,draw,circle,scale=0.275](9) at (0,9){};
\node[left](9.1,0) at (9){\small$4$};
\node[fill,draw,circle,scale=0.275](9r) at (2,9){};

\node[fill,draw,circle,scale=0.275](8) at (0,8){};
\node[left](8.1,0) at (8){\small$3$};
\node[fill,draw,circle,scale=0.275](8r) at (2,8){};

\node[fill,draw,circle,scale=0.275](7) at (0,7){};
\node[left](7.1,0) at (7){\small$2$};
\node[fill,draw,circle,scale=0.275](7r) at (2,7){};

\node[fill,draw,circle,scale=0.275](6) at (0,6){};
\node[left](6.1,0) at (6){\small$1$};
\node[fill,draw,circle,scale=0.275](6r) at (2,6){};

\node[fill,draw,circle,scale=0.275](5) at (0,5){};
\node[left](9.1,0) at (5){\small$0$};
\node[fill,draw,circle,scale=0.275](5r) at (2,5){};

\node[fill,draw,circle,scale=0.275](4) at (0,4){};
\node[left](8.1,0) at (4){\small$-1$};
\node[fill,draw,circle,scale=0.275](4r) at (2,4){};

\node[fill,draw,circle,scale=0.275](3) at (0,3){};
\node[left](3.1,0) at (3){\small$-2$};
\node[fill,draw,circle,scale=0.275](3r) at (2,3){};

\node[fill,draw,circle,scale=0.275](2) at (0,2){};
\node[left](2.1,0) at (2){\small$-3$};
\node[fill,draw,circle,scale=0.275](2r) at (2,2){};

\node at (9)[above=3pt]{$\vdots$};
\node at (9r)[above=3pt]{$\vdots$};
\node at (2)[below=-1pt]{$\vdots$};
\node at (2r)[below=-1pt]{$\vdots$};
\draw[-](9)--(9r);
\draw[-](8)--(8r);
\draw[-](7)--(7r);
\draw[-](6)--(7r);
\draw[-](5)--(6r);
\draw[-](4)--(5r);
\draw[-](3)--(5r);
\draw[-](2)--(2r);
\node at (1,0.32){$g$};
\end{tikzpicture}
\quad
\begin{tikzpicture}[scale=0.58]
\node[fill,draw,circle,scale=0.275](9) at (0,9){};
\node[left](9.1,0) at (9){\small$4$};
\node[fill,draw,circle,scale=0.275](9r) at (2,9){};
\node[fill,draw,circle,scale=0.275](9rr) at (4,9){};
\node[fill,draw,circle,scale=0.275](9rrr) at (6,9){};
\node[fill,draw,circle,scale=0.275](9rrrr) at (8,9){};

\node[fill,draw,circle,scale=0.275](8) at (0,8){};
\node[left](8.1,0) at (8){\small$3$};
\node[fill,draw,circle,scale=0.275](8r) at (2,8){};
\node[fill,draw,circle,scale=0.275](8rr) at (4,8){};
\node[fill,draw,circle,scale=0.275](8rrr) at (6,8){};
\node[fill,draw,circle,scale=0.275](8rrrr) at (8,8){};

\node[fill,draw,circle,scale=0.275](7) at (0,7){};
\node[left](7.1,0) at (7){\small$2$};
\node[fill,draw,circle,scale=0.275](7r) at (2,7){};
\node[fill,draw,circle,scale=0.275](7rr) at (4,7){};
\node[fill,draw,circle,scale=0.275](7rrr) at (6,7){};
\node[fill,draw,circle,scale=0.275](7rrrr) at (8,7){};

\node[fill,draw,circle,scale=0.275](6) at (0,6){};
\node[left](6.1,0) at (6){\small$1$};
\node[fill,draw,circle,scale=0.275](6r) at (2,6){};
\node[fill,draw,circle,scale=0.275](6rr) at (4,6){};
\node[fill,draw,circle,scale=0.275](6rrr) at (6,6){};
\node[fill,draw,circle,scale=0.275](6rrrr) at (8,6){};

\node[fill,draw,circle,scale=0.275](5) at (0,5){};
\node[left](9.1,0) at (5){\small$0$};
\node[fill,draw,circle,scale=0.275](5r) at (2,5){};
\node[fill,draw,circle,scale=0.275](5rr) at (4,5){};
\node[fill,draw,circle,scale=0.275](5rrr) at (6,5){};
\node[fill,draw,circle,scale=0.275](5rrrr) at (8,5){};

\node[fill,draw,circle,scale=0.275](4) at (0,4){};
\node[left](8.1,0) at (4){\small$-1$};
\node[fill,draw,circle,scale=0.275](4r) at (2,4){};
\node[fill,draw,circle,scale=0.275](4rr) at (4,4){};
\node[fill,draw,circle,scale=0.275](4rrr) at (6,4){};
\node[fill,draw,circle,scale=0.275](4rrrr) at (8,4){};

\node[fill,draw,circle,scale=0.275](3) at (0,3){};
\node[left](3.1,0) at (3){\small$-2$};
\node[fill,draw,circle,scale=0.275](3r) at (2,3){};
\node[fill,draw,circle,scale=0.275](3rr) at (4,3){};
\node[fill,draw,circle,scale=0.275](3rrr) at (6,3){};
\node[fill,draw,circle,scale=0.275](3rrrr) at (8,3){};

\node[fill,draw,circle,scale=0.275](2) at (0,2){};
\node[left](2.1,0) at (2){\small$-3$};
\node[fill,draw,circle,scale=0.275](2r) at (2,2){};
\node[fill,draw,circle,scale=0.275](2rr) at (4,2){};
\node[fill,draw,circle,scale=0.275](2rrr) at (6,2){};
\node[fill,draw,circle,scale=0.275](2rrrr) at (8,2){};

\node at (9)[above=3pt]{$\vdots$};
\node at (9r)[above=3pt]{$\vdots$};
\node at (9rr)[above=3pt]{$\vdots$};
\node at (9rrr)[above=3pt]{$\vdots$};
\node at (9rrrr)[above=3pt]{$\vdots$};

\node at (2)[below=-1pt]{$\vdots$};
\node at (2r)[below=-1pt]{$\vdots$};
\node at (2rr)[below=-1pt]{$\vdots$};
\node at (2rrr)[below=-1pt]{$\vdots$};
\node at (2rrrr)[below=-1pt]{$\vdots$};
\draw[-](9)--(9r);
\draw[-](8)--(8r);
\draw[-](7)--(7r);
\draw[-](6)--(7r);
\draw[-](5)--(5r);
\draw[-](4)--(4r);
\draw[-](3)--(3r);
\draw[-](2)--(2r);

\draw[-](9r)--(9rr);
\draw[-](8r)--(8rr);
\draw[-](7r)--(7rr);
\draw[-](6r)--(6rr);
\draw[-](5r)--(6rr);
\draw[-](4r)--(4rr);
\draw[-](3r)--(3rr);
\draw[-](2r)--(2rr);

\draw[-](9rr)--(9rrr);
\draw[-](8rr)--(8rrr);
\draw[-](7rr)--(7rrr);
\draw[-](6rr)--(6rrr);
\draw[-](5rr)--(5rrr);
\draw[-](4rr)--(4rrr);
\draw[-](3rr)--(4rrr);
\draw[-](2rr)--(2rrr);

\draw[-](9rrr)--(9rrrr);
\draw[-](8rrr)--(8rrrr);
\draw[-](7rrr)--(7rrrr);
\draw[-](6rrr)--(6rrrr);
\draw[-](5rrr)--(5rrrr);
\draw[-](4rrr)--(5rrrr);
\draw[-](3rrr)--(3rrrr);
\draw[-](2rrr)--(2rrrr);

\node at (1,0.4){$\cn{c}^{[1]}$};
\node at (3,0.25){$\cn{c}$};
\node at (5,0.4){$\cn{c}^{[-2]}$};
\node at (7,0.4){$\cn{c}^{[-1]}$};
\end{tikzpicture}
    \caption{Example for the proof of Lemma~\ref{l:Ffs-gen}}
    \label{fig:c-generatese-fs}
\end{figure}

\begin{proposition}
\label{p:Fn-LPn-gen} The following facts hold. 
\hfill
    \begin{enumerate}[label = \textup{(\roman*)}]
        \item Each non-identity element of $\mathrm{F}(\ZZ)$ with periodicity $n$  generates $\m{F}_n(\ZZ)$. 
        \item For each infinite $S\subseteq \ZZ^+$, $\bigvee_{n\in S}\vr(\m{F}_n(\ZZ)) = \va{DLP}$. 
    \end{enumerate}
\end{proposition}

\begin{proof}
(i) follows from Theorem~3.18 of \cite{GS} and (ii) is implicit in \cite{GG2}, as detailed in Remark~6.16 of \cite{GS}. 
\end{proof}

As  discussed in the introduction, every $n$-periodic $\ell$-pregroup embeds into a wreath product $\m{Aut}(\m{J}) \wr \m{F}_n(\ZZ)$, where $\m J$ is a chain. 
The notion of a wreath product of an $\ell$-group and an $\ell$-pregroup is  given in \cite{GG2}; here we give the definition for the special case where the $\ell$-group is an automorphism $\ell$-group.

For a chain $\m{J}$ and an $\ell$-pregroup $\m{L}$ we define the \emph{wreath product} $\m{Aut}(\m{J}) \wr \m{L} = ( \mathrm{Aut}(\m{J}) \times {L}^J, \mt, \jn,  \cdot, (id, \overline{1}), {}^{\ell}, {}^r )$, where we write $\overline{1}$ for the sequence $(1)_{i\in J}$, and for $f = (\widetilde{f}, \overline{f}), g = (\widetilde{g},\overline{g}) \in \mathrm{Aut}(\m{J}) \times {L}^J$ we let 
\begin{align*}
    (\widetilde{f},\overline{f}) \cdot (\widetilde{g},\overline{g}) &= (\widetilde{f}\circ \widetilde{g}, (\overline{f}\otimes \widetilde{g}) \cdot \overline{g}) \\
    (\widetilde{f},\overline{f})^\ell &= (\widetilde{f}^{-1}, \overline{f}^\ell \otimes \widetilde{f}^{-1}) \\
    (\widetilde{f},\overline{f})^r &= (\widetilde{f}^{-1}, \overline{f}^r \otimes \widetilde{f}^{-1}),
\end{align*}
where for $h\in \mathrm{Aut}(\m{J})$ and  $\overline{a}\in {L}^{J}$ we define the element  $\overline{a} \otimes h \in {L}^{J}$ by $(\overline{a} \otimes h)_i = \overline{a}_{h(i)}$; the order of $\m{Aut}(\m{J}) \wr \m{L}$ is defined by 
\[
(\widetilde{f},\overline{f}) \leq (\widetilde{g},\overline{g}) \quad \text{iff} \quad \widetilde{f} \leq \widetilde{g} \quad \text{and} \quad \widetilde{f}(i) = \widetilde{g}(i) \implies \overline{f}_i \leq \overline{g}_i, \text{ for all } i\in J.
\]
In particular, we have $\widetilde{f \mt g} = \widetilde{f} \mt \widetilde{g}$, $\widetilde{f \jn g} = \widetilde{f} \jn \widetilde{g}$, and for each $i\in J$
\[
\overline{(f \mt g)}_i = 
\begin{cases}
    \overline{f}_i &\text{if } \widetilde{f}(i) \leq \widetilde{g}(i) \\
    \overline{g}_i &\text{if } \widetilde{f}(i) \geq \widetilde{g}(i) \\
    \overline{f}_i \mt \overline{g}_i &\text{if } \widetilde{f}(i) = \widetilde{g}(i).
\end{cases} 
\quad
\overline{(f\jn g)}_i =
\begin{cases}
    \overline{f}_i &\text{if } \widetilde{f}(i) \geq \widetilde{g}(i) \\
    \overline{g}_i &\text{if } \widetilde{f}(i) \leq \widetilde{g}(i) \\
    \overline{f}_i \jn \overline{g}_i &\text{if } \widetilde{f}(i) = \widetilde{g}(i).
\end{cases} 
\]
For $f =(\widetilde{f},\overline{f}) \in \mathrm{Aut}(\m{J}) \times L^J$ we call $\widetilde{f}$ the \emph{global component} and $\overline{f}$ the \emph{local component} of $f$.

In \cite{GG2} it is shown that if $\m{J}$ is a chain and $\m{L}$ is an $\ell$-pregroup, then $\m{Aut}(\m{J})\wr \m{L}$ is an $\ell$-pregroup; if, in addition, $\m{L}$ is $n$-periodic, then so is $\m{Aut}(\m{J})\wr \m{L}$, for each $n \in \ZZ^+$.

Furthermore,  $\m{F}_n(\m{J}\overrightarrow{\times} \ZZ)$ is isomorphic to the wreath product $\m{Aut}(\m{J}) \wr \m{F}_n(\ZZ)$ and for $f\in \mathrm{F}_n(\m{J}\overrightarrow{\times} \ZZ)$ and corresponding  $(\widetilde{f},\overline{f}) \in \mathrm{Aut}(\m{J}) \times \mathrm{F}_n(\ZZ)^J$, we have that for all $(i,m) \in J\times \ZZ$: 
    \[
    f(i,m) = (\widetilde{f}(i), \overline{f}_i(m)).
    \]
For $f\in \mathrm{F}_n(\m{J}\overrightarrow{\times} \ZZ)$ and   $(\widetilde{f},\overline{f}) \in \mathrm{Aut}(\m{J}) \times \mathrm{F}_n(\ZZ)^J$ we will be writing $f \equiv (\widetilde{f},\overline{f})$ to indicate that the two elements correspond. 
Note that if $\m{L} \leq \m{F}_n(\m{J}\overrightarrow{\times} \ZZ)$, where $\m J$ is a chain, the set $\loc{L} = \{f \in L : \widetilde{f} = id \}$ of \emph{local} elements of $\m L$ forms a subalgebra $\loc{\m{L}}$, which we call  the \emph{local subalgebra} of $\m L$; where we note that $\loc{\m{L}}$ depends on the representation.

An $\ell$-pregroup is called \emph{properly} $n$-periodic if it is $n$-periodic and not $m$-periodic for any $m<n$.

\begin{proposition}\label{prop:period-subalg}
    If $\m{L}$ is a properly $n$-periodic $\ell$-pregroup, and  $p$ is a prime  with $p^k \mid n$, where $k \in \ZZ^+$, then $\m{L}$ contains a positive idempotent of periodicity divisible by $p^k$ and $\m{F}_{p^k}(\ZZ) \in \vr(\m{L})$. 
\end{proposition}
\begin{proof}
    By the representation theorem for periodic $\ell$-pregroups of \cite{GG2}, we may assume that $\m{L}\leq \m{F}_n(\m{J}\overrightarrow{\times}\ZZ)$ for some chain $\m{J}$. 
    By Remark~6.5 of \cite{GS}, also $\loc{\m{L}} \leq \m{F}_n(\ZZ)^J$ is properly $n$-periodic. Let $p$ be a prime with $p^k\mid n$. 
    Then there exists an element $a\in \loc{L}$ such that $p^k\mid \per(a)$,
    since $n =\lcm\{\per(b) : b\in \loc{L}\}$ because $\loc{\m{L}}$ is \emph{properly} $n$-periodic.
    Moreover, by Lemma~5.5 of \cite{GS}, we have $\per(a) =\per(\overline{a})= \lcm\{\per(\overline{a}_i) : i\in J\}$, so there exists a $j\in J$ such that $p^k \mid \per(\overline{a}_j)$. By Proposition~\ref{p:Fn-LPn-gen}(i), we have $\overline{a}_j \in \mathrm{F}_n(\ZZ)$, 
    so $\langle \overline{a}_j \rangle = \m{F}_m(\ZZ)$,
    where $m: = \per(\overline{a}_j)$. 
    
    For the first part of the statement,  given a positive idempotent $d$ of periodicity $m$ in $\langle \overline{a}_j \rangle = \m{F}_m(\ZZ)$, there exists a term $t$ such that $d=t(\overline{a}_j)$; so for $b:=t(a)$, we have
    $b\in \langle a \rangle$ and $\overline{b}_j=t(\overline{a}_j)=d$ is a positive idempotent of $\m{F}_m(\ZZ)$ of periodicity $m$; so $bb^\ell$ is a positive idempotent of $\m{L}$ with  $\overline{(bb^\ell)}_j=d d^\ell=d$, and
    \[
    p^k \mid m=\per(d)=\per(\overline{(bb^\ell)}_j) \mid \lcm\{\per(\overline{(bb^\ell)}_i) : i\in J\}=\per(bb^\ell).
    \]
    Now for the second part, since $p^k \mid m$  we have $\m{F}_{p^k}(\ZZ) \leq \m{F}_m(\ZZ)$; moreover, $\langle \overline{a}_j \rangle = \m{F}_m(\ZZ)$ is a subalgebra of the image of the $j$th projection of $\loc{\m{L}}\leq \m L$, thus $\m{F}_{p^k(\ZZ)} \in \vr(\m{L})$. 
\end{proof}

\begin{corollary}\label{cor:per-gen-DLP}
   If\/ $\m{L}$ is a distributive $\ell$-pregroup that contains elements of arbitrarily large periodicity, then it contains positive idempotent elements of arbitrarily large periodicity and $\vr(\m{L}) = \mathsf{DLP}$.
\end{corollary}
\begin{proof}
   If $\m{L}$ is a distributive $\ell$-pregroup that contains elements of arbitrarily big periodicity, then the set $Q$ of prime powers that divide these periodicities must be infinite.  
   By Proposition~\ref{prop:period-subalg} applied to the respective periodic subalgebras of $\m{L}$,  for each  $q \in Q$, $\m{L}$ contains positive idempotents of periodicity divisible by $q$ (hence the first claim follows) and  $\m{F}_{q}(\ZZ) \in \vr(\m{L})$. 
   By Proposition~\ref{p:Fn-LPn-gen}(ii), $\mathsf{DLP} = \bigvee_{q\in P} \vr(\m{F}_{q}(\ZZ)) \subseteq \vr(\m{L}) \subseteq \mathsf{DLP}$.
\end{proof}

\section{Composite functions and wreath products}\label{s: composite}

Recall that for $f\in \mathrm{F}_n(\m{J} \overrightarrow{\times} \ZZ)$, where $n \in \ZZ^+$ and $\m J$ is a chain, the global component of $f$ is an element of $\mathrm{Aut}(\m J)$. 
The analogue of this notion outside the $n$-periodic settings does not always yield an automorphism of $\m J$, and actually not even a function on $\m J$.

For $f\in \mathrm{F}(\m{J} \overrightarrow{\times} \ZZ)$ we define the following relation on $J$:
\[
\widetilde{f} := \{(i,j) : f(i,m) = (j,p) \text{ for some } m,p \in \ZZ\}.
\]
We call a map $f\in \mathrm{F}(\m{J}\overrightarrow{\times} \ZZ)$ \emph{composite} if $\widetilde{f}$ is a bijective map (in which case it is actually an automorphism of $\m J$) and \emph{local} if $\widetilde{f}= id_J$. 
Note that if $f$ is local, then there is a unique $\overline{f} \in \mathrm{F}(\ZZ)^J$ such that $f(i,k) = (i,\overline{f}_i(k))$ for all $i \in J$ and $k \in \ZZ$.  

\begin{remark}\label{r: idemptNOTcomposite}
We note that, in contrast to the $n$-periodic case (Lemma~5.6 of \cite{GS}), there are idempotent elements of $\mathrm{F}(\m{J} \overrightarrow{\times} \ZZ)$ that are not local. 
For example, we define the function $f$ on $\ZZ\overrightarrow{\times} \ZZ$ that maps the interval $[(2n,0), (2n+2, -1)]$ to the point $(2n+1,0)$, for all $n$. In this case $f$ is in $\mathrm{F}(\ZZ \overrightarrow{\times} \ZZ)$, as can be easily verified,  and  is idempotent, but it is not local as $\widetilde{f}$ is not even a function.  
\end{remark}

\begin{lemma}[\cite{GG2}]\label{l: periodic-is-composite}
Every periodic element of\/ $\mathrm{F}(\m{J}\overrightarrow{\times} \ZZ)$ is composite.
\end{lemma}

For $\m{L} \leq \m{F}(\m{J}\overrightarrow{\times} \ZZ)$
we denote by  $\comp{L}$ the set of composite maps in $L$ and by $\loc{L}$ the set of local maps in $L$; observe that the notion of locality for $n$-periodic $\ell$-pregroups is a special case of this definition.
Also, note that $\loc{L} \subseteq \comp{L}$ and that $\comp{L}$ and $\loc{L}$ give rise to subalgebras $\comp{\m{L}}$ and $\loc{\m{L}}$ of $\m{L}$; if $\m L= \m{F}(\m{J}\overrightarrow{\times} \ZZ)$, we denote the corresponding subalgebras by $\comp{\m{F}}(\m{J}\overrightarrow{\times} \ZZ)$ and $\loc{\m{F}}(\m{J}\overrightarrow{\times}\ZZ)$. 
In particular, we get the following result. 

\begin{lemma}\label{l:comp-is-wreath}
If $\m{J}$ is a chain, then $\comp{\m{F}}(\m{J}\overrightarrow{\times} \ZZ) \cong \m{Aut}(\m{J}) \wr \m{F}(\ZZ)$ via the isomorphism
        \[
        \f\colon \m{Aut}(\m{J}) \wr \m{F}(\ZZ) \to \comp{\m{F}}(\m{J}\overrightarrow{\times} \ZZ), \f(\widetilde{f},\overline{f})(i,n) = (\widetilde{f}(i),\overline{f}_i(n)).
        \]
In particular $\varphi$ restricts to an isomorphism between $\{id\} \times \m{F}(\ZZ)^J$ and $\loc{\m{F}}(\m{J}\overrightarrow{\times}\ZZ)$, yielding that $\loc{\m{F}}(\m{J}\overrightarrow{\times}\ZZ) \cong \m{F}(\ZZ)^J$, and each $f\in\loc{\mathrm{F}}(\m{J}\overrightarrow{\times}\ZZ)$ is finite to one.
\end{lemma}

\begin{proof}
    It is straightforward to check that $\f$ is an injective homomorphism. So it remains to show that $\f$ is surjective. 
    For $f\in \comp{\mathrm{F}}(\m{J} \overrightarrow{\times} \ZZ)$ we define the map $\widehat{f} \in \comp{\mathrm{F}}(\m{J} \overrightarrow{\times} \ZZ)$, by $\widehat{f}(i,k) = (\widetilde{f}(i), k)$ and observe that it is an order-preserving bijection. 
    Note that $\widehat{f}^{-1}f$ is local and let $\overline{f}=(\overline{f}_i)_{i \in J} \in \mathrm{F}(\ZZ)^J$ be such that $\widehat{f}^{-1}f(i,k) = (i,\overline{f}_i(k))$  for all $i \in J$ and $k \in \ZZ$. 
    Then 
    \begin{equation}\label{eq:psi}
        f(i,k) = \widehat{f}\widehat{f}^{-1}f(i,k) = \widehat{f}(i,\overline{f}_i(k)) = (\widetilde{f}(i),\overline{f}_i(k)).
    \end{equation}
    Now, we define the map $\ps \colon \comp{\mathrm{F}}(\m{J} \overrightarrow{\times} \ZZ) \to \mathrm{Aut}(\m{J}) \times \mathrm{F}(\ZZ)^J$, by $\ps(f) = (\widetilde{f}, \overline{f})$; 
    note that, since $\widetilde{f}$ is order-preserving,  $\ps(f)$ in an element of the codomain. 
    By \eqref{eq:psi} we have $\f\ps = id$, so $\f$ is surjective.

    For the second part of the lemma, note that every
    $f \in \loc{\mathrm{F}(\m{J} \overrightarrow{\times} \ZZ)}$ is a set-theoretic disjoint union of maps $f_j \in \mathrm{F}(\{j\} \times \ZZ)$, where $j \in J$. 
    Since the elements of $\m{F}(\{j\}\times \ZZ) \cong \m{F}(\ZZ)$ are finite to one maps, the same holds for $f$.
\end{proof}

This isomorphism restricts to the isomorphism between $\m{Aut}(\m{J}) \wr \m{F}_n(\ZZ)$ and $\m{F}_n(\m{J}\overrightarrow{\times} \ZZ)$ described in the preceeding section.  

The next lemma gives a sufficient condition for $\widetilde{f}$ to be a surjective map.

\begin{lemma} \label{l:minmax-surj-dichot}
    For each $f\in \mathrm{F}(\m{J}\overrightarrow{\times} \ZZ)$,
    if for all $j \in J$, $f[J \times \ZZ] \cap (\{j\} \times \ZZ)$ has no minimum and no maximum, then the relation $\widetilde{f}$ is a surjective map.
\end{lemma}
\begin{proof}
     We first show that $\widetilde{f}$ is a map.
     If not, then there exist $i,j_1,j_2\in J$ and $x_1,x_2,y_1,y_2 \in \ZZ$, such that $f(i,x_1) = (j_1,y_1)$, $f(i,x_2) = (j_2,y_2)$ and $j_1 < j_2$; since $\ZZ$ is discrete and $f$ is order-preserving, we may choose $x_1$ and $x_2$ such that $x_1 +1 = x_2$, which leads to the contradiction $(j_1,y_1) = \max (f[J\times \ZZ] \cap (\{j_1\} \times \ZZ))$. 
     
     Now, if $\widetilde{f}$ is not surjective, then there exists an $i\in J$ such that $i\notin \widetilde{f}[J]$, so for each $y \in  \{i\}\times \ZZ$ and $(k,m)\in f[J\times \ZZ]$, $y\leq (k,m)$ iff $i < \widetilde{f}(k)$. 
     Note that for all $y \in  \{i\}\times \ZZ$,  $ff^\ell(y)$ is the smallest element of the image of $f$ above $y$, as $ff^\ell$ is a closure operator with the same image as that of $f$, so $ff^\ell(y)=ff^\ell(y')$, for all  $y, y' \in  \{i\}\times \ZZ$.
     Hence there exists a $j\in J$ with  $i<j$ and $n\in \ZZ$ such that for each $y \in \{i\} \times \ZZ$, $ff^\ell(y) = (j,n)$. 
     But then, for each $m\in \ZZ$ with $m<n$, $(j,m) \notin f[J\times \ZZ]$, i.e., $(j,n) = \min (f[J \times \ZZ] \cap (\{j\} \times \ZZ))$, contradicting our assumption.
\end{proof}

\section{The dichotomy}\label{s: dichotomy}

To prove the main theorem we will consider an arbitrary non-periodic distributive $\ell$-pregroup $\m L$ and show that it generates all of $\mathsf{DLP}$. In this section, we show that $\m L$ either contains a non-periodic idempotent or it contains idempotents of arbitrary big periodicity. This will allow us to consider only these two cases when proving the main theorem. We first prove the result for 1-generated subalgebras of $\mathrm{F}(\ZZ)$ and then for the general case.

\begin{lemma}\label{l:FZ non-periodic}
   If $f$ is a non-periodic element of $\mathrm{F}(\ZZ)$, then $\langle f \rangle$ either contains a non-periodic idempotent or it contains idempotents of arbitrary big periodicity.
\end{lemma}
\begin{proof}
    Suppose that all idempotents in $\langle f \rangle$ are periodic and suppose for a contradiction that the periodicity of the idempotents is bounded. 
    Let $m$ be the maximum periodicity of an idempotent element in $\langle f \rangle$.  
    Since $f$ is non-periodic, it is not invertible; as a result $ff^{\ell}$ is not the identity 
   so
    $ff^{\ell}$ is a strictly positive idempotent, thus $m>1$. 
    Since $f$ is not $m$-periodic we may assume without loss of generality that $f(0) + m \neq f(m)$. 
    
    If $f(m)-f(0) < m$, then we consider the idempotent $\bot_1$ of periodicity $m$ that maps the interval $[0,m-1]$ to $0$, $[m,2m-1]$ to $m$ etc., and  $\bot_2$ an idempotent of periodicity $m$ that maps the interval $[f(0),f(m)]$ to some point $k$ (note that such a map exists since $f(m) - f(0) <m$). 
    As $\langle f\rangle$ contains a map of periodicity $m$ by hypothesis, it contains all $m$-periodic functions of $\m F(\ZZ)$, by Proposition~\ref{p:Fn-LPn-gen}(i),  
    so $\bot_1,\bot_2 \in \langle f\rangle$.
    Moreover, the map $g = \bot_2f\bot_1$ maps all the elements in the interval $[0,2m-1]$ to $k$ and the map $g^{\ell}g$ is an idempotent mapping $[0,2m-1]$ to one point. 
    Hence by assumption $g^{\ell}g$ has periodicity at least $2m>m$ and it is an element of $\langle f\rangle$ contradicting the maximality of $m$. 

    Otherwise $f(m)-f(0) > m$; we consider $h = f\bot_1$. Then $h(m) -h(0) = f(m) - f(0) > m$ and $h^\ell(x) = m$ for all  $x \in [h(0)+1, h(m)]$, 
    So, $h^\ell$ maps $h(m)-h(0)$ many elements to the same element and so does the idempotent $hh^\ell$, which thus must have periodicity at least $h(m) -h(0) > m$, contradicting the maximality of $m$. 
\end{proof}

The following lemma shows that for every strictly positive/negative order-preserving map on $\ZZ$ there exist intervals of all sizes, which $f$ does not properly expand.

\begin{lemma}\label{l:m-interval}
    For each  order-preserving map $f\colon \ZZ \to \ZZ$ with $f<1$ or $f>1$ and each $m\in \NN$, there exists an $x\in \ZZ$ such that 
    $f(x+m) \leq f(x) + m$.
\end{lemma}
\begin{proof}
    Let $m\in \NN$ and $f\in \mathrm{F}(\ZZ)$ with $f>1$, and suppose for a contradiction that no such $x$ exists, i.e., for all $y\in \ZZ$, $f(y) -f(y-m) > m$. 
    We define the sequence $(x_k)_{k\in \NN}$  by $x_k = f(-km) +km$. Since $f>1$, we get $x_k \geq 0$ for each $k\in \NN$. 
    Moreover, for each $k\in \NN$ we have
    \[
        f(-km) -f(-(k+1)m) = f(-km) - f(-km -m)  > m,
    \]
    yielding 
    \begin{align*}
        x_k = f(-km) + km > f(-(k+1)m) + (k+1)m = x_{k+1}.
    \end{align*}
    So $(x_k)_{k\in \NN}$ is strictly decreasing but bounded from bellow by $0$, a contradiction. 
    The proof for  $f<1$ is obtained similarly, by constructing a strictly increasing sequence.
\end{proof}

\begin{lemma}\label{l:nonper-dichot}
     If $m\geq 1$ and $f$ is a strictly positive function in $\mathrm{F}(\ZZ)$ that is not $m$-periodic, then there exists $x\in \ZZ$ such that $f(x+m) < f(x)+m$ or $f^\ell(x+m) < f^\ell(x)+m$.
\end{lemma}

\begin{proof}
    Since $f$ is not $m$-periodic there exists $x\in \ZZ$ such that $f(x+m) \neq f(x)+m$. 
    If $f(x + m) < f(x) +m$ fails, then $f(x+m) \geq f(x) +m +1$, so we have 
    \[
    f^\ell(f(x)+1 + m) =  \min \{ y \in \ZZ : f(x) + 1 + m \leq f(y) \} \leq x+m
    \]
    and
    \[
    f^\ell(f(x) + 1) =  \min \{ y \in \ZZ : f(x) + 1 \leq f(y) \} > x.
    \]
    Thus for $x' = f(x) + 1$, we have
    \[
    f^\ell(x' + m) =  f^\ell(f(x) + 1 + m) \leq x +m < f^\ell(f(x) + 1) + m = f^\ell(x') + m.\qedhere
    \]
\end{proof}

\begin{lemma}\label{l:non-per-shrink}
      If $m\geq 1$ and $f$ is a strictly positive function in $\mathrm{F}(\ZZ)$ that is not $m$-periodic, then there exist  $k,x\in \ZZ$, such that $f^{[k]}f^{\ell}(x+m) < f^{[k]}f^{\ell}(x) + m$.
\end{lemma}

\begin{proof}
    Since $f>1$ is not $m$-periodic, by Lemma~\ref{l:nonper-dichot}, there exists an $x\in \ZZ$ such that $f(x+m) < f(x) + m$ or $f^\ell(x +m) < f^\ell(x) +m$. 
    
    We first assume that $f(x+m) < f(x) + m$. 
    Since $f>1$, we get $f^\ell <1$, so  by Lemma~\ref{l:m-interval}, there exists a $y\in \ZZ$ such that $f^\ell(y + m) \leq f^\ell(y) + m$.  
    For $k := f^{\ell}(y)- x$, we have $f^\ell(y) -k = x$, so
    \begin{align*}
        f^{[k]}(f^\ell(y + m)) &\leq f^{[k]}(f^\ell(y) + m) \\
          &= f(f^\ell(y) -k + m) +k \\
          &= f(x + m) + k \\
          &< f(x) +k + m \\
          &= f(f^\ell(y) -k) + k + m \\
          &= f^{[k]}(f^\ell(y)) + m.
    \end{align*}
     
    Now, if $f^\ell(x+m) < f^\ell(x) + m$, then $m' := f^\ell(x+m) - f^\ell(x) <m$. 
    By Lemma~\ref{l:m-interval}, there exists a $z\in \ZZ$ such that $f(z+m') \leq f(z) + m'$. For $k := f^\ell(x) -z$, we have $f^\ell(x) - k = z$, so
    \begin{align*}
        f^{[k]}(f^\ell(x+m)) &= f^{[k]}(f^\ell(x+m) + f^\ell(x) - f^\ell(x)) \\
        &= f(f^\ell(x) - k + m') +k\\
        &= f(z + m') +k \\
        &\leq f(z) + m' + k \\
        &= f(f^\ell(x) -k) + k + m' \\
        &= f^{[k]}(f^\ell(x)) + m'\\
        &< f^{[k]}(f^\ell(x)) + m. \qedhere
    \end{align*}
\end{proof}

\begin{proposition}\label{p:non-per-gen}
    If $f\in \mathrm{F}(\m{J}\overrightarrow{\times}\ZZ)$ is strictly positive and non-periodic, and $\widetilde{f}$ is a bijective map, then $\langle f\rangle$ contains a local non-periodic positive idempotent element or local periodic positive idempotent elements of arbitrarily large periodicity.
\end{proposition}

\begin{proof}
    Since $\widetilde{f}$ is a bijective map, by Lemma~\ref{l:comp-is-wreath}, $f$ can be seen as a member $(\widetilde{f},\overline{f})$ of $\m{Aut}(\m{J})\wr \m{F}(\ZZ)$. 
    Since $f$ is non-periodic and $f^{[k]}= (\widetilde{f},\overline{f}^{[k]})$ for every $k \in \ZZ$, we get that $\overline{f}$ is non-periodic; 
    so for each $m\geq 1$ there exists an $i\in J$, such that $\overline{f}_i$ is not $m$-periodic. 
    Since $\overline{f}_i \in \mathrm{F}(\ZZ)$, by Lemma~\ref{l:non-per-shrink}, there exist $k,x\in \ZZ$ such that $\overline{f}_i^{[k]}\overline{f}_i^\ell(x+m) < \overline{f}_i^{[k]}\overline{f}_i^\ell(x) +m$. 
    Also, $f^{[k]}f^\ell = (\widetilde{f},\overline{f}^{[k]})(\widetilde{f}^{-1},\overline{f}^\ell\otimes \widetilde{f}^{-1}) 
    =(id, (\overline{f}^{[k]} \otimes \widetilde{f}^{-1}) \cdot (\overline{f}^\ell\otimes \widetilde{f}^{-1})) = (id,(\overline{f}^{[k]}\overline{f}^\ell)\otimes \widetilde{f}^{-1})$. 
    So $g := f^{[k]}f^\ell$ is local and for $j=\widetilde{f}(i)$, we have 
    \[
    \overline{g}_j(x+m) =  \overline{f}_{\widetilde{f}^{-1}(j)}^{[k]}\overline{f}_{\widetilde{f}^{-1}(j)}^\ell(x+m) = \overline{f}_i^{[k]}\overline{f}_i^\ell(x+m) < \overline{f}_i^{[k]}\overline{f}_i^\ell(x) +m = \overline{g}_j(x) + m.
    \]
    Hence $\overline{g}_j$ is not $m$-periodic.  
    There are two cases.
    If $\overline{g}_j$ is non-periodic, by Lemma~\ref{l:FZ non-periodic}, $\overline{g}_j$ generates a non-periodic positive idempotent or positive idempotents of arbitrarily large periodicity. 
    Otherwise, $\overline{g}_j$ is periodic, say $k$-periodic with $k\nmid m$, and generates a $k$-periodic positive idempotent, by Proposition~\ref{p:Fn-LPn-gen}(i). 

    In both cases there exists a positive idempotent $a\in \langle \overline{g}_j \rangle$ which is not $m$-periodic so there is a term $t$ such that $a=t(\overline{g}_j)$; 
    we define $h:=t(g)$, so $h\in \langle g \rangle$ and $hh^\ell \in \langle g\rangle$. Since $g$ is local,  $h$ and $hh^\ell$ are also local; so $hh^\ell \in \langle g\rangle$ is a local positive idempotent with $\overline{hh^\ell}_j =aa^\ell= a$. 
    In particular, since $a$ is not $m$-periodic, also $hh^\ell$ is not $m$-periodic. 
    In conclusion for each $m\geq 1$ we get a local positive idempotent that is not $m$-periodic, yielding either a non-periodic local positive idempotent or local periodic positive idempotents of arbitrarily large periodicity. 
\end{proof}

\begin{lemma}\label{l:pos-idem-non-periodic}
    If
    \(\m L\)\/ is a non-periodic distributive $\ell$-pregroup,  then it contains a non-periodic positive idempotent or periodic positive idempotents of arbitrarily big periodicity. 
\end{lemma}

\begin{proof}
    If every element of $\m{L}$ is periodic then, since  $\m{L}$ is not periodic, it contains periodic elements of arbitrary big periodicity and thus, by Corollary~\ref{cor:per-gen-DLP}, positive idempotents of arbitrary big periodicity. 
   
   Otherwise there is 
   a non-periodic $f\in L$; if both $f \jn 1$ and $f^\ell \jn 1$ were periodic, then they would also be periodic  with period equal to the least common multiple, say $n$, of the individual periods, so $f = (f \jn 1)(f \mt 1)=(f \jn 1)(f^\ell \jn 1)^r$ would also be $n$-periodic. 
   Thus at least one of the positive functions $f \jn 1$ and $f^\ell \jn 1$ is not periodic; as a result we may assume that $f>1$.

   Since $\m L$ is distributive, up to isomorphism it has a representation of the form $\m{L} \leq \m{F}(\m{J} \overrightarrow{\times} \ZZ)$.

    (1) If, for some $j \in J$, $f[J \times \ZZ] \cap (\{j\} \times \ZZ)$ has a minimum element, say $y$, then $y-1$ is not in the image of $f$ and $z:=ff^r(y-1) \leq y-1$, so  $z \notin f[J \times \ZZ] \cap (\{j\} \times \ZZ)$; also, $z \in \{k\} \times \ZZ$, for some $k<j$. 
    Now, for $w:=f^r(y-1)$, we have $f(w)=z$ and $f(w+1)=y$, where $z \in \{k\} \times \ZZ$ and $y \in \{j\} \times \ZZ$, yielding $f^\ell(z) \leq w$ and $w+1 = f^\ell(z+1)$.
    Therefore, $ff^\ell$ sends two elements of the same component, namely $z$ and $z+1$, 
    to different components of  $\m J \times \ZZ$, because $ff^\ell(z)=z$ and $ff^\ell(z+1)=y$, hence $ff^\ell$ is not composite and thus also non-local. 
    By Lemma~\ref{l: periodic-is-composite}, 
    $ff^\ell$ is not periodic, while we also know that it is positive and  idempotent. 

    (2) If, for some $j \in J$,  $f[J \times \ZZ] \cap (\{j\} \times \ZZ)$ has a maximum element, then we can replace $f$ by $f^\ell$ and proceed as in  (1).

    (3) If, for all $j \in J$, $f[J \times \ZZ] \cap (\{j\} \times \ZZ)$ has no minimum and no maximum, then
    by Lemma~\ref{l:minmax-surj-dichot}, $\widetilde{f}$ is a surjective map. 
    There are two cases:

    (3i) If $\widetilde{f}$ is injective and thus bijective, then, by Proposition~\ref{p:non-per-gen}, the claim follows.
    
    (3ii) If $\widetilde{f}$ is not injective,
    then there exist
    $f(x), f(y) \in \{j\} \times \ZZ$ for some $j$ and $x  \in \{k\} \times \ZZ$, $y  \in \{m\} \times \ZZ$ with $k < m$. Since $f$ is order-preserving and $\{j\} \times \ZZ$ discrete, we get that $\{f(a): a \in  \{k\} \times \ZZ\}$ has a maximum element in $\{j\} \times \ZZ$, say $z$. 
    Then $f^\ell(z) \in  \{s\} \times \ZZ$, where $s\leq k<m$. Also, $f^\ell(z+1) \in \{t\} \times \ZZ$, where  $t>k$. 
    Thus $f^{\ell}f^{\ell\ell}$ maps $f^\ell(z)$ to $f^{\ell}f^{\ell\ell}f^\ell(z)=f^\ell(z)$, an element of $\{s\} \times \ZZ$, and $f^\ell(z)+1$ to 
    \[
    f^\ell f^{\ell\ell}(f^\ell(z)+1) = f^\ell(f(f^\ell(z))+1) = f^\ell(z+1),
    \]
    an element of $\{t\} \times \ZZ$;  hence it is not composite and thus also non-local.
    By Lemma~\ref{l: periodic-is-composite}, 
    $f^{\ell}f^{\ell\ell}$ is not periodic, while it is positive and  idempotent. 
\end{proof}

\section{Simulating functions}\label{s: simulation}

In this section we define the key notion of a simulation, which in effect codifies the ability of a distributive $\ell$-pregroup to falsify equations that fail in another distributive $\ell$-pregroup.

For $f, g \in \mathrm{F}(\ZZ)$ and $k\in \NN$, we write $f \sim_k g$ if $f$ and $g$ agree on $[-k,k]$; note that $\sim_k$ is an equivalence relation on $\mathrm{F}(\ZZ)$.

For $f\in \mathrm{F}(\ZZ)$ and $k\in \NN$ we define
\begin{align*}
        a(f,k) &= \max ( \{ \lvert f^\ell(k) \rvert, \lvert f^\ell(-k) \rvert+1 \}) \\
        b(f,k) &= \max (\{ \rvert f^r(k) \rvert +1, \rvert f^r(-k) \rvert \}) \\
        c(f,k) &= \max\{ \lvert f(k) \rvert, \lvert f(-k) \rvert  \}.
    \end{align*}

\begin{lemma}\label{l:simulation_connectives}
    Let $f_1,f_2,g_1,g_2 \in F(\ZZ)$, $k\in \NN$ and 
    \begin{align*}
        a = a(f_1,k), \quad
        b = b(f_1,k), \quad
        c = c(f_1,k).
    \end{align*}
    \begin{enumerate}[label = \textup{(\roman*)}]
        \item If $f_1 \sim_k g_1$ and $f_2 \sim_k g_2$, then $f_1 \bullet f_2 \sim_k g_1 \bullet g_2$ for $\bullet \in \{ \mt, \jn \}$.
        \item If $f_1 \sim_a g_1$, then  $f_1^\ell \sim_k g_1^\ell$.
        \item If $f_1 \sim_b  g_1$, then  $f_1^r \sim_k g_1^r$.
        \item If $f_1 \sim_k g_1$ and $f_2 \sim_c g_2$, then $f_2 \circ f_1 \sim_k g_2 \circ g_1$.
    \end{enumerate}
\end{lemma}

\begin{proof}
    (i) For $x\in [-k,k]$, $f_1(x) = g_1(x)$ and $f_2(x) = g_2(x)$, so $(f_1\mt f_2)(x) = \min(f_1(x),f_2(x)) = \min(g_1(x),g_2(x)) = (g_1 \mt g_2)(x)$, and similarly $(f_1\jn f_2)(x) = (g_1 \jn g_2)(x)$.
    
    (ii) Let $x\in [-k,k]$. 
    Then $f^\ell_1(x) \in [-a,a]$, so $x \leq f_1(f^\ell_1(x)) = g_1(f^\ell_1(x))$, hence $g^\ell_1(x)\leq f^\ell_1(x)$.
    For $z:=\max\{g^\ell_1(x), -a\}$, we have $z\in [-a,a]$ and $x \leq g_1(g^\ell_1(x)) \leq g_1(z) = f_1(z)$, hence $f^\ell_1(x)\leq z=\max\{g^\ell_1(x), -a\}$; since $-a < f^\ell_1(-k) \leq f^\ell_1(x)$ we get $f^\ell_1(x) \leq g_1^\ell(x)$, as well.
   
    (iii) Let $x\in [-k,k]$. Then $f^r_1(x) \in [-a,a]$, so $g_1(f^r_1(x)) = f_1(f^r_1(x))\leq x$, hence $f^r_1(x) \leq g^r_1(x)$.
    For $z:=\min\{g^r_1(x), a\}$, we have $z\in [-a,a]$ and $f_1(z) = g_1(z)  \leq g_1(g^r_1(x)) \leq x$, hence $ \min\{g^r_1(x), a\} = z \leq f^r_1(x)$; since $f^r_1(x) \leq f^r_1(k) < a$ we get $g_1^r(x) \leq f^r_1(x)$, as well.
    
    (iv) For $x\in [-k,k]$, we have $f_1(x) = g_1(x) \in [-c,c]$, so $f_2(f_1(x)) = g_2(g_1(x))$.
\end{proof}

For $f\in \mathrm{F}(\ZZ)$ and $\m{L} \leq \m{F}(\ZZ)$, we say that $\m{L}$ \emph{simulates} $f$ and write $f \precsim \m{L}$ if for each $k\in \NN$, there exists a $g\in L$ such that $f\sim_k g$.  
For $\m{L},\m{M} \leq \m{F}(\ZZ)$, we say that $\m{L}$ \emph{simulates} $\m{M}$ and write $\m{M} \precsim \m{L}$ if $\m{L}$ simulates each $h \in M$. 

For $f\in \mathrm{F}(\ZZ)$, $\m{M}\leq \m{F}(\ZZ)$ and $\m{L} \leq \m{F}(\m{J}\overrightarrow{\times} \ZZ)$, we say that $\m{L}$ \emph{simulates $f$ at $j\in J$} and write $f \precsim^j \m{L}$ if for each $k\in \NN$, there exists $g\in \loc{L}$ such that $\overline{g}_j \sim_k f$ (or short $g\sim_k^j f$); we say that $\m{L}$ \emph{simulates $\m{M}$ at $j\in J$} and write $\m{L} \precsim^j \m{M}$ if $\m{L}$ simulates each $h\in M$ at $j$. 
Finally, for a distributive $\ell$-pregroup $\m L$ and for $\m{M}\leq \m{F}(\ZZ)$ we write $\m{M} \precsim \m{L}$ and say that $\m L$ \emph{simulates} $\m  M$, if there exist a representation of $\m L$ with global chain $\m J$ and $j\in J$ such that $\m{M}\precsim^j \m{L}$.

\begin{lemma}\label{l:simulation_diagram}
    Let $\m{L} \leq \m{F}(\m{J} \overrightarrow{\times} \ZZ)$,  $j\in J$,  $f_1,\dots, f_n \in \mathrm{F}(\ZZ)$ and let $t(x_1,\dots, x_n)$ be a term. 
    If $\m{L}$ simulates $f_1,\dots, f_n$ at $j$, then  
    for each  $k \in \NN$, there exist $g_1,\dots, g_n \in \loc{L}$ such that $t(g_1,\dots, g_n) \sim_k^j t(f_1,\dots, f_n)$; thus $\m{L}$ simulates $t(f_1,\dots, f_n)$ at $j$.
\end{lemma}

\begin{proof}
    Abusing the notation slightly we will write $\sim_k$ for $\sim_k^j$.
    We will perform term induction on $t$, but first we explain what the maps $g_1, \ldots, g_n$ are taken to be; $g_i$ is obtained by the fact that $\m L$ simulates $f_i$ for $k_i$ (i.e., $g_i \sim_{k_i} f_i$), where the value of $k_i$ is defined below and depends on $k$ and $f_1,\ldots, f_n$. 
    We compute the value of $k_i$ by first assigning values $k(s) \in \NN$ for every occurrence of a subterm $s$ of $t$ recursively, guided by the values in Lemma~\ref{l:simulation_connectives}; 
    we adopt the simple notation $k(s)$ even though the value depends not (only) on $s$ but on its position in the term tree of $t$. For the root, $t$, we set $k(t):=k$.  
    If $t=t_1 \mt t_2$ or $t=t_1 \jn t_2$, then $k(t_1):=k(t)$ and $k(t_2):=k(t)$. 
    If $t=s^\ell$, then $k(s)=a(s(f_1, \ldots, f_n), k(t))$. If $t=s^r$, then $k(s):=b(s(f_1, \ldots, f_n), k(t))$. 
    If $t=t_1 \cdot t_2$, then $k(t_1):=k(t)$ and  $k(t_2):=c(s(f_1, \ldots, f_n), k(t))$. 
    We stress that $k(s)$ is assigned to all the nodes $s$ in the term tree of $t$, so $k(s)$ may have different values for different occurrences of $s$ in the term tree of $t$; we define $k_i$ to be the maximum of the values $k(x_i)$, where the variable $x_i$ appears (at a leaf node) in the term tree of $t$.  
    
    Now we prove by induction (starting from the leafs) that for each node $s$, $s(g_1,\dots,g_n)\sim_{k(s)} s(f_1,\dots,f_n)$.
    If $s=x_i$ is a leaf (hence a variable), then $k(s) = k_i$ and it follows from the assumption that there is a $g_i$ such that $g_i \sim_{k_i} f_i$. 
   If $s=s_1 \cdot s_2$, then by the inductive hypothesis, $s_1(g_1,\dots,g_n)\sim_{k(s_1)} s_1(f_1,\dots,f_n)$ and $s_2(g_1,\dots,g_n)\sim_{k(s_2)} s_2(f_1,\dots,f_n)$. 
   Also, by definition $k(s_1)=k(s)$ and  $k(s_2)=c(s(f_1, \ldots, f_n), k(s))$, so by Lemma~\ref{l:simulation_connectives} we get $s(g_1,\dots,g_n)\sim_{k(s)} s(f_1,\dots,f_n)$. 
   The inductive step for the other connectives is very similar.
    \end{proof}

\begin{corollary}\label{c: simgen}
If a distributive $\ell$-pregroup $\m L$ 
simulates a generator of\/ $\m M \leq \m{F}(\ZZ)$, 
then $\m L$ simulates $\m M$.        
\end{corollary}

\begin{proof}
Let $f$ be a generator of $\m M$ and let $g \in M$. 
Then there exists a term $t_g$ such that $g=t_g(f)$, hence $\m L$ simulates $g=t_g(f)$, by Lemma~\ref{l:simulation_diagram}.  
\end{proof}

\begin{proposition}\label{p: simvar}
        If a distributive $\ell$-pregroup $\m L$  simulates  $\m{M}\leq \m{F}(\ZZ)$, then $\vr(\m{M}) \subseteq \vr(\m{L})$, i.e., if\/ $\m{M}\precsim \m{L}$, then $\vr(\m{M}) \subseteq \vr(\m{L})$.
\end{proposition}

\begin{proof}
    If an equation $1 \leq t(x_1,\dots, x_n)$ fails in $\m{M}$, then there exist $f_1,\dots, f_n \in M$ and $m \in \ZZ$ such that $m \nleq t(f_1,\dots, f_n)(m)$, yielding 
    \[
    0 < t(f_1,\dots,f_n)(m) -m = t(f_1,\dots,f_n)^{[-m]}(0) = t(f_1^{[-m]},\dots,f_n^{[-m]})(0),
    \] 
    since $h\mapsto h^{[-m]}$ is an automorphism. Hence by considering $f_i^{[-m]}$ instead of $f_i$ we can assume $m=0$. 
    Now, by Lemma~\ref{l:simulation_diagram}, there exist $g_1,\dots, g_n \in \loc{L}$ (for a representation $\m{L}\leq \m{F}(\m{J}\overrightarrow{\times}\ZZ)$) such that $t(g_1,\dots, g_n) \sim_1^j t(f_1,\dots, f_n)$, so $(j,0) \nleq t(g_1,\dots, g_n)(j,0)$, i.e., $1 \leq t(x_1,\dots, x_n)$ fails in $\m L$.
\end{proof}

\begin{lemma}\label{l:dlp-simulation}
Let $\m{L}$ be a distributive $\ell$-pregroup.
\begin{enumerate}[label = \textup{(\roman*)}]
    \item If $\cc\precsim \m{L}$, then $\vr(\m{L}) = \va{DLP}$.
    \item If $f\in  \mathrm{F}(\ZZ)$, $\per(f) = n$, $f\neq 1$, and $f\precsim \m{L}$, then $\vr(\m{F}_n(\ZZ))\subseteq \vr(\m{L})$. 
\end{enumerate}
\end{lemma}

\begin{proof}
(i) By Proposition~\ref{l:Ffs-gen}, $\m{F}_\mathrm{fs}(\ZZ)$ is  generated by $\cc$, so, by Corollary~\ref{c: simgen},  $\m{F}_\mathrm{fs}(\ZZ) \precsim \m{L}$. Therefore, by   Proposition~\ref{p: simvar}, $\vr(\m{F}_\mathrm{fs}(\ZZ)) \subseteq \vr(\m{L})$. Since  $\vr(\m{F}_\mathrm{fs}(\ZZ)) = \va{DLP}$ by Proposition~\ref{p:DLP-gen-fs}, we get $\vr(\m{L}) = \va{DLP}$.

(ii) As in (i), using Proposition~\ref{p:Fn-LPn-gen}(i) and Corollary~\ref{c: simgen}, we get that $\m L$ simulates $\m{F}_n(\ZZ)$, hence $\vr(\m{F}_n(\ZZ))\subseteq \vr(\m{L})$ by Proposition~\ref{p: simvar}.
\end{proof}

\section{Proper subvarieties of \texorpdfstring{$\mathsf{DLP}$}{DLP}}\label{s: main}

In this section we establish the main result that every non-periodic distributive $\ell$-pregroup generates all of $\mathsf{DLP}$, by using the dichotomy established in Lemma~\ref{l:pos-idem-non-periodic}. One of the two cases of the dichotomy (arbitrary big periodicities) can already be handled with results of  Section~\ref{s: prelims}; here we rely on the notion of simulation and develop the auxiliary results needed for the other case. 

Note that for $x = (i,k) \in J\times \ZZ$ and $n\in \ZZ$, we have $x +n := (i,k+n)$. 
Also, we define the term $\slt(y) := y^{\ell\ell\ell}y \jn 1$.
The \emph{maximal height} of a local element $f\in \mathrm{F}(\m{J} \overrightarrow{\times} \ZZ)$ is the maximum of the set $\{\overline{f}_j(k)-k : j \in J, k \in \ZZ\}$, if it exists.

We will use the following lemma right away for the case where $\m J$ is a singleton, and also later for arbitrary $\m J$.

\begin{lemma}\label{l:pos-idem-slope1}
    If $f\in \mathrm{F}(\m{J} \overrightarrow{\times} \ZZ)$ is a strictly positive idempotent, then $\delta(f)$ is a local element satisfying
    \[
    \delta(f)(x) = \begin{cases}
        x +1 & \text{if } x < f(x) \text{ and } f^{\ell}(f(x)) = x, \\
        x & \text{otherwise.}
    \end{cases}
    \]
    In particular, $\delta(f)$ is strictly positive idempotent with maximal height $1$. 
\end{lemma}

\begin{proof}
We first show that
    for each $x\in J\times \ZZ$,
     \begin{enumerate}[label = \textup{(\roman*)}]
            \item If $f(x) \neq f^\ell(x)$, then $(f^{\ell\ell\ell}f)(x) = f^\ell(x) + 1$, so  $\slt(f)(x)\leq x+1$.
            \item If $f(x) = f^\ell(x)$, then $(f^{\ell\ell\ell}f)(x) \leq x$, so $\slt(f)(x)=x$.
            \end{enumerate}
The proof is the same as the proof a weaker claim (for $\m J$ a singleton) in the proof of Lemma~3.21 of \cite{GS}; for completeness, we repeat the proof, noting that it works for arbitrary $\m J$. 
First note that since $f$ is positive and idempotent, we have $f f^\ell=f$ and $f^\ell f=f^\ell$, by Remark~2.2 of \cite{GS}; 
so for all $x \in J \times \ZZ$, we have $f^\ell(x)= f^\ell(f(x)) =\min f^{-1}[f(x)]$, hence also $f^\ell(x)\leq x \leq f(x)$.

For (i), if $f(x) \neq f^\ell(x)$ then
    $f^\ell(x) < f(x)$, hence $f^\ell(x) \leq  f(x)-1$; 
    so $f(x) =f(f^\ell(x))\leq f(f(x)- 1) \leq f(f(x))=f(x)$, thus $f(f(x)- 1) = f(x)$.
    Therefore, 
    \[
    f^{\ell\ell\ell}(f(x)) = f^\ell(f(x)-1) + 1 = f^\ell f(f(x)-1) + 1 = f^\ell f(x) + 1= f^\ell(x) + 1.
    \]
    Finally, since $f^\ell(x)\leq x$, we get $(f^{\ell\ell\ell}f)(x) = f^\ell(x) + 1\leq x +1$.
        
    For (ii), note that if  $f(x) = f^\ell(x)$, then $f^\ell(x)\leq x \leq f(x)$ yields $f(x) = x$.
    Hence, since $f^\ell \leq 1$,
    \[
    f^{\ell\ell\ell}(f(x)) = f^{\ell\ell\ell}(x) = f^{\ell}(x-1) + 1 \leq x-1 + 1 = x.
    \]
This shows (i) and (ii). 
Now we are ready to show the lemma; let $x\in J\times \ZZ$. 

If $x < f(x)$ and $f^\ell(f(x)) = x$, then $f^\ell(x) = f^\ell(f(x)) = x$, so $f^\ell(x) < f(x)$; by (i), we get $(f^{\ell\ell\ell}f)(x) = f^\ell(x) + 1 = x +1$, i.e., $\delta(f)(x) = x +1$. 
Otherwise, we have $x\geq f(x)$ or $f^\ell(f(x)) \neq x$; we distinguish further cases. 

If $f^\ell(x) = f^\ell(f(x)) \neq x$, then, by the positivity of $f$, we get $f^\ell(x) < x$, i.e., $f^\ell(x) + 1 \leq x$, so in either case (i) or (ii) yields $\delta(f)(x) = x$. 

If $f^\ell(x) = f^\ell(f(x)) = x$ and $x\geq f(x)$, then $f(x) = x$, by the positivity of $f$, hence $f(x) = x = f^\ell(x)$; by (ii) we get $\delta(f)(x) = x$.
\end{proof}

\begin{lemma}\label{l:simulate_fs_eventually_const}
    Let $f$ be a strictly positive element of\/ $\mathrm{F}(\ZZ)$.
    \begin{enumerate}[label = \textup{(\roman*)}]
        \item If there is a $k\in \ZZ$ such that $f(l) = l$ for all  $l> k$, then $\cc \precsim \langle f \rangle$.
        \item If there is a $k\in \ZZ$ such that $f(l) = l$ for all  $l < k$, then $\cc \precsim \langle f \rangle$.
    \end{enumerate}
\end{lemma}
\begin{proof}
    (i) By Lemma~\ref{l:pos-idem-slope1}, $g:= \slt(f)$ has maximal height $1$ and is positive idempotent. We first show that $g(l+1) = l+1$ for all $l>k$. Indeed, for each $l>k$,
    \[
    f^\ell(l) = \min \{ x \in \ZZ : l\leq f(x) \} \leq l,
    \]
    hence 
    \[
    f^{\ell\ell\ell}(f(l+1)) = f^{\ell\ell\ell}(l+1) = f^\ell(l+1 -1) +1 \leq l+1,
    \]
    i.e., $g(l+1) = l+1$ for each $l>k$. 
    So by considering a suitable shift of $g$ (i.e., $g^{[m]}$ for some $m\in \ZZ$) we may assume that 
    $g(l+1) = l+1$ for each $l>0$ and $g(0) \neq 0$; since $g$ is positive, we get $g(0) = 1$, and since it  has maximal height $1$ and is positive idempotent we get $g(-1) = -1$; hence $g^\ell(1) = 0$. 
    Now for each $n\geq 2$ we define $g_n = (g^\ell)^{[-n]}(g^\ell)^{[-(n-1)]}\cdots (g^\ell)^{[-2]}$ and let $x\in [-n,n]$. For $x\geq 0$,  we have 
    \[
    g_ng(x) = g^{[-n]}g^{[-(n-1)]}\cdots g^{[-2]}g(x) = g(x),
    \]
    while for $x < 0$, we get
    \begin{align*}
            g_ng(x) &\leq g_ng(-1) \\
            &= (g^\ell)^{[-n]}(g^\ell)^{[-(n-1)]}\cdots (g^\ell)^{[-2]}g(-1)\\
            &= (g^\ell)^{[-n]}(g^\ell)^{[-(n-1)]}\cdots (g^\ell)^{[-2]}(-1)\\
            &= (g^\ell)^{[-n]}(g^\ell)^{[-(n-1)]}\cdots (g^\ell)^{[-3]}(g^\ell(1) -2) \\
            &= (g^\ell)^{[-n]}(g^\ell)^{[-(n-1)]}\cdots (g^\ell)^{[-3]}(-2) \\
            &\ \  \vdots\\
            &=-n \leq x.
    \end{align*}
   Thus for $h := g_ng \jn 1$ we have for $x\in [-n,n]$,
    \[
    h(x) = \begin{cases}
        1 & \text{if } x = 0 \\
        x & \text{if } x \neq 0,
    \end{cases}
    \]
    yielding $h\sim_n c$;  see Figure~\ref{fig:placeholder}. Thus $c \precsim \langle f \rangle$.
   
    The proof of (ii) is similar.
\end{proof}

\begin{figure}
    \centering
    \begin{tikzpicture}[scale=0.58]
\node[fill,draw,circle,scale=0.275](9) at (0,9){};
\node[left](9.1,0) at (9){\small$5$};
\node[fill,draw,circle,scale=0.275](9r) at (2,9){};

\node[fill,draw,circle,scale=0.275](8) at (0,8){};
\node[left](8.1,0) at (8){\small$4$};
\node[fill,draw,circle,scale=0.275](8r) at (2,8){};

\node[fill,draw,circle,scale=0.275](7) at (0,7){};
\node[left](7.1,0) at (7){\small$3$};
\node[fill,draw,circle,scale=0.275](7r) at (2,7){};

\node[fill,draw,circle,scale=0.275](6) at (0,6){};
\node[left](6.1,0) at (6){\small$2$};
\node[fill,draw,circle,scale=0.275](6r) at (2,6){};

\node[fill,draw,circle,scale=0.275](5) at (0,5){};
\node[left](9.1,0) at (5){\small$1$};
\node[fill,draw,circle,scale=0.275](5r) at (2,5){};

\node[fill,draw,circle,scale=0.275](4) at (0,4){};
\node[left](8.1,0) at (4){\small$0$};
\node[fill,draw,circle,scale=0.275](4r) at (2,4){};

\node[fill,draw,circle,scale=0.275](3) at (0,3){};
\node[left](3.1,0) at (3){\small$-1$};
\node[fill,draw,circle,scale=0.275](3r) at (2,3){};

\node[fill,draw,circle,scale=0.275](2) at (0,2){};
\node[left](2.1,0) at (2){\small$-2$};
\node[fill,draw,circle,scale=0.275](2r) at (2,2){};
\node at (9)[above=3pt]{$\vdots$};
\node at (9r)[above=3pt]{$\vdots$};
\node at (2)[below=-1pt]{$\vdots$};
\node at (2r)[below=-1pt]{$\vdots$};
\draw[-](9)--(9r);
\draw[-](8)--(9r);
\draw[-](7)--(7r);
\draw[-](6)--(7r);
\draw[-](5)--(5r);
\draw[-](4)--(5r);
\draw[-](3)--(3r);
\draw[-](2)--(2r);
\node at (1,0.4){$g$};
\end{tikzpicture}%
\quad
\begin{tikzpicture}[scale=0.58]
\node[fill,draw,circle,scale=0.275](9) at (0,9){};
\node[left](9.1,0) at (9){\small$5$};
\node[fill,draw,circle,scale=0.275](9r) at (2,9){};

\node[fill,draw,circle,scale=0.275](8) at (0,8){};
\node[left](8.1,0) at (8){\small$4$};
\node[fill,draw,circle,scale=0.275](8r) at (2,8){};

\node[fill,draw,circle,scale=0.275](7) at (0,7){};
\node[left](7.1,0) at (7){\small$3$};
\node[fill,draw,circle,scale=0.275](7r) at (2,7){};

\node[fill,draw,circle,scale=0.275](6) at (0,6){};
\node[left](6.1,0) at (6){\small$2$};
\node[fill,draw,circle,scale=0.275](6r) at (2,6){};

\node[fill,draw,circle,scale=0.275](5) at (0,5){};
\node[left](9.1,0) at (5){\small$1$};
\node[fill,draw,circle,scale=0.275](5r) at (2,5){};

\node[fill,draw,circle,scale=0.275](4) at (0,4){};
\node[left](8.1,0) at (4){\small$0$};
\node[fill,draw,circle,scale=0.275](4r) at (2,4){};

\node[fill,draw,circle,scale=0.275](3) at (0,3){};
\node[left](3.1,0) at (3){\small$-1$};
\node[fill,draw,circle,scale=0.275](3r) at (2,3){};

\node[fill,draw,circle,scale=0.275](2) at (0,2){};
\node[left](2.1,0) at (2){\small$-2$};
\node[fill,draw,circle,scale=0.275](2r) at (2,2){};
\node at (9)[above=3pt]{$\vdots$};
\node at (9r)[above=3pt]{$\vdots$};
\node at (2)[below=-1pt]{$\vdots$};
\node at (2r)[below=-1pt]{$\vdots$};
\draw[-](9)--(8r);
\draw[-](8)--(8r);
\draw[-](7)--(6r);
\draw[-](6)--(6r);
\draw[-](5)--(5r);
\draw[-](4)--(4r);
\draw[-](3)--(3r);
\draw[-](2)--(2r);
\draw[-](9)--(8r);
\draw[-](8)--(8r);
\draw[-](7)--(6r);
\draw[-](6)--(6r);
\node at (1,0.55){$(g^\ell)^{[2]}$};
\end{tikzpicture}%
\quad
\begin{tikzpicture}[scale=0.58]
\node[fill,draw,circle,scale=0.275](9) at (0,9){};
\node[left](9.1,0) at (9){\small$5$};
\node[fill,draw,circle,scale=0.275](9r) at (2,9){};

\node[fill,draw,circle,scale=0.275](8) at (0,8){};
\node[left](8.1,0) at (8){\small$4$};
\node[fill,draw,circle,scale=0.275](8r) at (2,8){};

\node[fill,draw,circle,scale=0.275](7) at (0,7){};
\node[left](7.1,0) at (7){\small$3$};
\node[fill,draw,circle,scale=0.275](7r) at (2,7){};

\node[fill,draw,circle,scale=0.275](6) at (0,6){};
\node[left](6.1,0) at (6){\small$2$};
\node[fill,draw,circle,scale=0.275](6r) at (2,6){};

\node[fill,draw,circle,scale=0.275](5) at (0,5){};
\node[left](9.1,0) at (5){\small$1$};
\node[fill,draw,circle,scale=0.275](5r) at (2,5){};

\node[fill,draw,circle,scale=0.275](4) at (0,4){};
\node[left](8.1,0) at (4){\small$0$};
\node[fill,draw,circle,scale=0.275](4r) at (2,4){};

\node[fill,draw,circle,scale=0.275](3) at (0,3){};
\node[left](3.1,0) at (3){\small$-1$};
\node[fill,draw,circle,scale=0.275](3r) at (2,3){};

\node[fill,draw,circle,scale=0.275](2) at (0,2){};
\node[left](2.1,0) at (2){\small$-2$};
\node[fill,draw,circle,scale=0.275](2r) at (2,2){};

\node at (9)[above=3pt]{$\vdots$};
\node at (9r)[above=3pt]{$\vdots$};
\node at (2)[below=-1pt]{$\vdots$};
\node at (2r)[below=-1pt]{$\vdots$};
\draw[-](9)--(8r);
\draw[-](8)--(8r);
\draw[-](7)-- (6r);
\draw[-](6)--(6r);
\draw[-](5)--(5r);
\draw[-](4)--(5r);
\draw[-](3)--(3r);
\draw[-](2)--(2r);

\node at (1,0.55){$(g^{\ell})^{[2]}\circ g$};
\end{tikzpicture}%
\quad
\begin{tikzpicture}[scale=0.58]
\node[fill,draw,circle,scale=0.275](9) at (0,9){};
\node[left](9.1,0) at (9){\small$5$};
\node[fill,draw,circle,scale=0.275](9r) at (2,9){};

\node[fill,draw,circle,scale=0.275](8) at (0,8){};
\node[left](8.1,0) at (8){\small$4$};
\node[fill,draw,circle,scale=0.275](8r) at (2,8){};

\node[fill,draw,circle,scale=0.275](7) at (0,7){};
\node[left](7.1,0) at (7){\small$3$};
\node[fill,draw,circle,scale=0.275](7r) at (2,7){};

\node[fill,draw,circle,scale=0.275](6) at (0,6){};
\node[left](6.1,0) at (6){\small$2$};
\node[fill,draw,circle,scale=0.275](6r) at (2,6){};

\node[fill,draw,circle,scale=0.275](5) at (0,5){};
\node[left](9.1,0) at (5){\small$1$};
\node[fill,draw,circle,scale=0.275](5r) at (2,5){};

\node[fill,draw,circle,scale=0.275](4) at (0,4){};
\node[left](8.1,0) at (4){\small$0$};
\node[fill,draw,circle,scale=0.275](4r) at (2,4){};

\node[fill,draw,circle,scale=0.275](3) at (0,3){};
\node[left](3.1,0) at (3){\small$-1$};
\node[fill,draw,circle,scale=0.275](3r) at (2,3){};

\node[fill,draw,circle,scale=0.275](2) at (0,2){};
\node[left](2.1,0) at (2){\small$-2$};
\node[fill,draw,circle,scale=0.275](2r) at (2,2){};

\node at (9)[above=3pt]{$\vdots$};
\node at (9r)[above=3pt]{$\vdots$};
\node at (2)[below=-1pt]{$\vdots$};
\node at (2r)[below=-1pt]{$\vdots$};
\draw[-](9)--(9r);
\draw[-](8)--(8r);
\draw[-](7)-- (7r);
\draw[-](6)--(6r);
\draw[-](5)--(5r);
\draw[-](4)--(5r);
\draw[-](3)--(3r);
\draw[-](2)--(2r);

\node at (1,0.4){ $\cn{c}$};
\end{tikzpicture}%
    \caption{Example for the proof of Lemma~\ref{l:simulate_fs_eventually_const}}
    \label{fig:placeholder}
\end{figure}
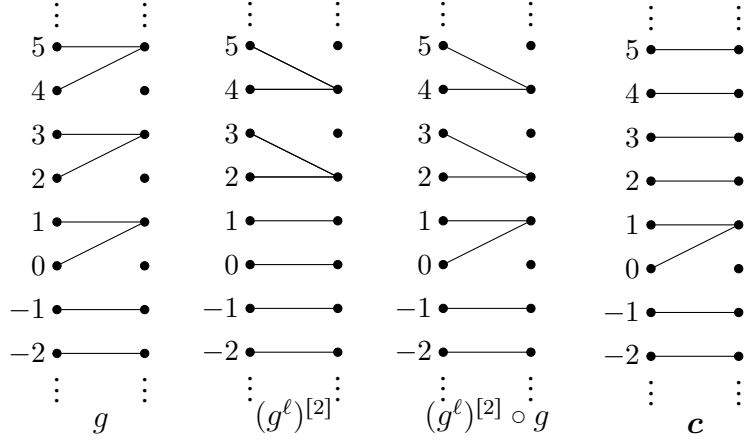

\begin{lemma}\label{l:non-per-gen}
    For any non-periodic positive idempotent $f\in \mathrm{F}(\ZZ)$, $\vr(\langle f \rangle) = \mathsf{DLP}$. 
\end{lemma}

\begin{proof}
Let $\m{L} = \langle f \rangle$ and define $F = \{f^{[k]} : f^{[k]}(0) > 0 \text{ and } k\in \ZZ \}$ and $g = \bigwedge F$. 
Then, since $f$ is strictly positive and idempotent, $g$ is positive and idempotent and, since there exists a $k\in \ZZ$ with $f(k) = k+1$, we get $g(0) = 1$ and $g(1) = 1$; so $g$ is strictly positive. 

If for all $l < 0$, $g(l) = l$ or for all $l > 0$, $g(l) = l$, then, since $\m{L}$ simulates $g$ we get that $\m{L}$  simulates $\cc$, by Lemma~\ref{l:simulate_fs_eventually_const}, and $\vr(\m{L}) = \va{DLP}$, by Lemma~\ref{l:dlp-simulation}(i). 

Otherwise, there exist an $l<0$ with $g(l) > l$ and $l'>0$ with $g(l')>l'$. 
Hence we get $l\leq g(l)-1 < g(l)$ which implies $g(l) \leq g(g(l)-1) \leq g(g(l)) = g(l)$, i.e., $g(g(l)-1) = g(l)$. Similarly, it follows that  $g(g(l')-1) = g(l')$. 
Therefore, there exist a maximal $d \in \ZZ$ with $d< 0$ and $g(d) =d+1$ and a minimal $u \in \ZZ$ with $0<u$ and $g(u) = u+1$.
We will show that this case leads to a contradiction.

For $m = \max\{u, -d  \}$ there exists a finite subset $F_1 \subseteq F$ such that for $g_1 :=  \bigwedge F_1$ we have the equality of the restrictions $(g_1)_{[-m,m]} = g_{[-m,m]}$.  

We will now show that $u=-d$. If $u<-d$, let $F_2 := \{h^{[-d]} : h \in F_1 \}$ and $g_2 := \bigwedge F_2 = g_1^{[-d]}$. Then $g_2(u) = g_1(u+d) - d = u$, since $d < u+d < 0$. 
Hence $(\bigwedge (F_1 \cup F_2))(u) = (g_1 \mt g_2)(u) = \min \{u+1, u\} = u$, yielding $g(u) = u$, a contradiction. 
Similarly the assumption $u>-d$ yields a contradiction. So we have $u=-d$.

Finally, we derive a contradiction by showing that $f$ is $u$-periodic; let $k\in \ZZ$. If $f(k) =k$ and $f(k+u) >k+u$, then $f^{[-k-u]}(0) = f(k+u)-k-u>0$, so $f^{[-k-u]} \in F$; but also 
\[
f^{[-k-u]}(d) = f^{[-k-u]}(-u) = f(k) - k-u = -u = d,
\]
so  $g(d) = d$, a contradiction. 
If $f(k) > k$ and $f(k+u) = k+u$, then $f^{[-k]}(0) = f(k)-k > 0$, so $f^{[-k]} \in F$; but also $f^{[-k]}(u) = f(k+u) -k = u$, so $g(u) = u$, a contradiction. 
So we have shown that for all $k\in \ZZ$,
\[
f(k) = k \iff f(k+u) = k+u,
\]
i.e., the set of fixed points of $f$ is closed under translation by $u$ (the set is $u$-periodic).
But, since $f$ is positive and idempotent, i.e., a closure operator on $\ZZ$, the fixed points of $f$ completely determine $f$ (for each $k\in \ZZ$, $f(k) = \min \{ m\in \ZZ : k \leq m = f(m)\}$); hence  $f$ is $u$-periodic.
\end{proof}

\begin{theorem}\label{t: main}
    If $\m L$ is a non-periodic distributive $\ell$-pregroup,
    then $\vr(\m{L}) = \va{DLP}$. 
\end{theorem}
\begin{proof}
 Since $\m L$ is distributive,  
by \cite{GG1} we obtain a representation of the form $\m{L} \leq \m{F}(\m{J} \overrightarrow{\times} \ZZ)$ up to isomorphism.
    Since $\m{L}$ is not periodic, by Lemma~\ref{l:pos-idem-non-periodic}, $\m{L}$ contains a positive non-periodic idempotent or periodic idempotents of arbitrary big periodicity. 
    In the latter case we get $\vr(\m{L}) = \mathsf{DLP}$ by Corollary~\ref{cor:per-gen-DLP}.
    
    So we may assume that $\m L$ contains a non-periodic positive idempotent $f$, hence $f$ is actually strictly positive. There are three cases
    \begin{enumerate}
        \item $f$ is not local, i.e., there exists a $j\in J$ such that $\widetilde{f}(j) \neq \{j\}$.
        \item $f$ is local and $\overline{f}_j$ is periodic for each $j\in J$.
        \item $f$ is local and for some $j\in J$, $\overline{f}_j$ is not periodic. 
    \end{enumerate}

In the first case, since $f$ is positive, there exists $j_+ \in \widetilde{f}(j)$ with $j_+ > j$. 
Now, after renaming the elements of the copies of $\ZZ$, we may assume that $f(j,0) = (j_+,0) = f(j_+,0)$ and that $(j,0) = \min\{(i,k) : f(i,k) = (j_+,0) \} = f^\ell(j_+,0)$. 
So, by Lemma~\ref{l:pos-idem-slope1}, $h := f^{\ell\ell\ell}f \jn 1$ is a strictly positive local idempotent such that $h(j,0) = (j,1)$ and $h(j,k) = (j,k)$ for each $k>0$. 
By Lemma~\ref{l:simulate_fs_eventually_const}, $\m{L}$ simulates $\cc$ at $j$, so $\vr(\m{L}) = \va{DLP}$ by Lemma~\ref{l:dlp-simulation}(i).

In the second case we have $f = (id,\overline{f})$. If $P = \{\per(\overline{f}_j) :  j\in J  \}$ were finite, i.e., $P = \{n_1,\dots,n_k\}$, then by Lemma~5.5 of \cite{GS} $f$ would be $\lcm(n_1,\dots, n_k)$-periodic, a contradiction; thus $P$ is infinite. 
For each $n\in P$, there exists $j_n\in J$ such that $\per(\overline{f}_{j_n})=n$, so $\langle \overline{f}_{j_n}\rangle \cong \m{F}_n(\ZZ)$, by Proposition~\ref{p:Fn-LPn-gen}(i). 
Therefore, since $\m{F}_n(\ZZ)$ is a subalgebra of the $j_n$th projection of $\loc{\m{L}}$, we have 
$\m{F}_n(\ZZ) \in \vr(\m{L})$ for each $n\in P$; as $P$ is infinite,  by Proposition~\ref{p:Fn-LPn-gen}(ii) 
it follows that $\vr(\m{L}) = \va{DLP}$.

In the third case we  have $\langle \overline{f}_j \rangle \in \vr(\m{L})$, since $\langle \overline{f}_j \rangle$ is a subalgebra of the $j$th projection of $\loc{\m{L}} \leq \m L$. Hence $\vr(\m{L}) = \mathsf{DLP}$, by Lemma~\ref{l:non-per-gen}.
\end{proof}

\begin{corollary}\label{c: subvarDLP}
Every proper subvariety of \(\mathsf{DLP}\) is contained in \(\mathsf{LP}_n\) for some $n$.    
\end{corollary}

\begin{remark}
    From Corollary~\ref{c: subvarDLP} we obtain an alternative proof of the fact that $\mathsf{DLP}$ has a decidable equational theory, since for each equation $\varepsilon$, in parallel, we can try to prove it from the finitely many axioms of $\mathsf{DLP}$ and for each $n\geq 1$, by dovetailing, try to prove $n$-periodicity using the axioms of $\mathsf{DLP}$ together with $\varepsilon$. 
    This procedure will eventually stop, since either $\varepsilon$ holds in $\mathsf{DLP}$ and we get a proof of it, or it does not hold in $\mathsf{DLP}$ so there is a proof of $n$-periodicity from $\varepsilon$ for some $n\geq 1$.
\end{remark}

\begin{remark} 
    In \cite{GS} the subvariety lattice of periodic varieties of $\ell$-pregroups is studied and by Corollary~\ref{c: subvarDLP} this constitutes exactly the proper subvarieties of $\mathsf{DLP}$. 
    As a result, the subvariety lattice of $\mathsf{DLP}$, excluding $\mathsf{DLP}$ itself, decomposes into a disjoint union of bounded intervals each consisting of all properly $n$-periodic varieties, for different $n \in \ZZ^+$; see Figure~\ref{f: SubvarLat}. 
    The top of such an interval is the variety $\mathsf{LP}_n$ and the bottom is characterized in \cite{GS} as $\bigvee_{q\in Q_n} \vr(\m{F}_{q}(\ZZ))$, where $Q_n = \{p_1^{k_1},\dots, p_{l}^{k_l} \}$ and $n = p_1^{k_1}\cdots p_l^{k_l}$ is the prime decomposition of $n$;
    if $n$ is a prime power, this interval is $[\vr(\m{F}_n(\ZZ)), \mathsf{LP}_n]$.

    The tops of these intervals, the varieties $\mathsf{LP}_n$, are ordered simply by divisibility. The varieties of the form $\vr(\m{F}_n(\ZZ))$ are also ordered by divisibility, and in general, as shown in \cite{GS}, the varieties of the form $\bigvee_{s\in S} \vr(\m{F}_{s}(\ZZ))$, where $S \subseteq \ZZ^+$,
    (without $\mathsf{DLP}$ itself) form an ideal in the subvariety lattice of $\mathsf{DLP}$ isomorphic to the lattice of finite downsets of the divisibility lattice.\footnote{Furthermore, \cite{GS} provides explicit axiomatizations for all of these varieties.} The bottoms of the intervals, are also ordered by divisibility: $\bigvee_{q\in Q_m} \vr(\m{F}_{q}(\ZZ)) \subseteq  \bigvee_{q\in Q_n} \vr(\m{F}_{q}(\ZZ))$ iff every element of $Q_m$ divides some element of $Q_n$ iff $m$ divides $n$.
    
\end{remark}

\begin{figure}[h]
    \centering
\begin{tikzpicture}[ place2/.style={circle,draw=black,fill=black, minimum size = 0pt, inner sep = 1.5ex, opacity = 0},scale=1.5]

  \node[place2] (LP6) at (-0.5,5) {};

  \node[place2] (mid) at (-1,4.1) {};
  
  \node[place2] (F6) at (1,3) {};
  \node[place2] (F62) at (0.85,3.05) {};
  \node[place2] (F2-3) at (1,2) {};
  \node[place2] (F2-32) at (0.7,2) {};

  \node[place2] (LP2) at (-1.5,3.5) {};
  
  \node[place2] (F3) at (1.75,1)  {};
  \node[place2] (F2) at (0.2,1)  {};
  \node[place2] (F22) at (0.15,1.05)  {};

  \node[place2] (LG) at (-2.85,1.95)  {};
  \node[place2] (LG2) at (-2.5,2)  {};
  
  \node[place2] (Ab4) at (-0.25,0.05) {};
  \node[place2] (Ab3) at (-0.25,0) {};
  \node[place2] (Ab2) at (-0.35,0) {};
  \node[place2] (Ab) at (0.39,-0.025)  {};
  \node[place2] (bot) at (-0.25,-0.6) {};
  
  \node[] at (LP6)  {\small$\mathsf{LP}_6$};
  
  \node[] at (F6)  {\small$\vr(\m{F}_6(\ZZ))$};
  \node[] at (F2-3)  {\small$\vr(\m{F}_2(\ZZ),\m{F}_3(\ZZ))$};
  \node[] at (LP2)  {\small$\mathsf{LP}_2$};

  \node[] at (F3)  {\small$\vr(\m{F}_3(\ZZ))$};
  \node[] at (F2)  {\small$\vr(\m{F}_2(\ZZ))$};

  \node[] at (LG)  {\small$\mathsf{LP}_1 =\mathsf{LG}$};

  \node[] at (Ab)  {\small$\mathsf{AbLG} = \vr(\m{F}_1(\ZZ))$};
  \node[] at (bot)  {\small$0$};

 \draw (bot) -- (Ab3);
 \draw (Ab4) -- (F3);
 \draw (Ab3) -- (F2);
 \draw (F2) -- (F2-3);
 \draw (F3) -- (F2-3);
 \draw (F2-3) -- (F6);

 \draw (Ab2) to[out=110,in=-20] (LG2);
 \draw (Ab2) to[out=160,in=275] (LG2);

 \draw (F22) to[out=100,in=-25] (LP2);
 \draw (F22) to[out=155,in=270] (LP2);

 \draw (F62) to[out=100,in=-25] (LP6);
 \draw (F62) to[out=155,in=270] (LP6);

 \draw (F2-32) to[out=100,in=-30] (mid);
 \draw (F2-32) to[out=155,in=290] (mid);

 \draw (LG2) to[out=30,in=-110] (LP2);
 \draw (LG2) to[out=85,in=205] (LP2);

 \draw (LP2) to[out=30,in=-105] (LP6);
 \draw (LP2) to[out=85,in=205] (LP6);
  
\end{tikzpicture}
    \caption{Disjoint intervals in the lattice of subvarieties of $\mathsf{DLP}$.}
    \label{f: SubvarLat}
\end{figure}
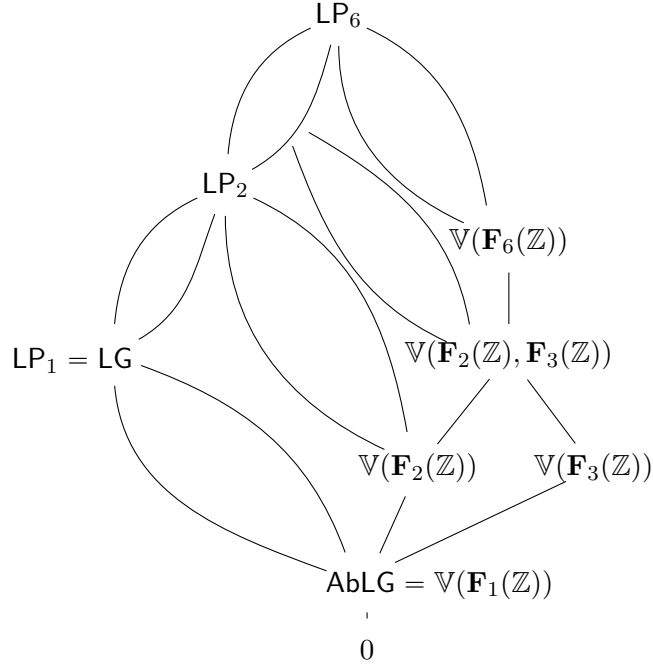

\section{An application of the main theorem}\label{s: app}
For $n\in \ZZ^+$ in \cite{GS} we defined the terms
\[
\g_n(x) =  x^{[0]} \mt  x^{[1]} \mt \cdots \mt x^{[n-1]}
\text{ and }
\h_n(x) =  x^{[0]} \jn  x^{[1]} \jn \cdots \jn x^{[n-1]}.
\]
There it is shown that if an $\ell$-pregroup $\m{L}$ is $n$-periodic then, for each $a\in L$, $\h_n(a)$ is the smallest invertible element above $a$ and $\g_n(a)$ is the largest invertible element below $a$. 
Moreover, the image of either of the two term functions is exactly the set of invertible elements of $\m L$ and it supports a subalgebra $\Grp{\m L}$ of $\m L$; note that $\Grp{\m L}$ is an $\ell$-group. 
These terms allow us to transfer certain properties of  $\Grp{\m L}$ to $\m L$ itself. 
For example, $\Grp{\m L}$ and $\m L$ have isomorphic lattices of convex subuniverses.

Note that if $\m{L}$ is a  $k$-periodic $\ell$-pregroup with $k\leq n$, then  $\g_n(a)$ and $\h_n(a)$ are invertible for each $a\in L$. 
However, it was left open in \cite[Remark~2.5]{GS} whether the converse also holds.  
This problem can now be resolved using Corollary~\ref{c: subvarDLP} and some information  about the structure of  subdirectly irreducible periodic $\ell$-pregroups, which we develop in the following lemmas.

Let $\m{L}\leq \m{Aut}(\m{J})\wr \m{F}_n(\ZZ)$.  
For an element $f\in \loc{L}$, we  set $\supp(f) := \{j\in J : \overline{f}_j \neq 1\}$ and we  define the \emph{local support} of $\m{L}$ as $\loc{\supp}(\m{L}) := \bigcup\{ \supp(f) : f\in \loc{L}\}$. 
Moreover, the \emph{global group} of $\m L$ is the subgroup $\glob{\m{L}}\leq \m{Aut}(\m{J})$ with universe 
\[
\glob{L} = \{\widetilde{f} : (\widetilde{f},\overline{f}) \in L\} = \{\widetilde{f} : (\widetilde{f},\overline{f}) \in \Grp{L}\}.
\] 

A subset $I$ of $J$ is called \emph{ $\glob{\m{L}}$-stable}, if $g(i) \in I$ for all $g \in \glob{L}$ and $i \in I$.
For a $\glob{\m{L}}$-stable subset $I\subseteq J$, we define the map 
\[
\pi_I\colon L\to  \mathrm{Aut}(\m{I})\times\mathrm{F}_n(\ZZ)^{I}, \quad
    \pi_I(\widetilde{f},(\overline{f}_j)_{j\in J}) = (\widetilde{f}|_{I},(\overline{f}_j)_{j\in I}).
\]
Note that the stability of $I$ guarantees that $\pi_I$ is well-defined.

\begin{lemma}\label{lemma:pi_I-hom}
    Let $\m{L}\leq \m{Aut}(\m{J})\wr \m{F}_n(\ZZ)$ and let $I\subseteq J$ be $\glob{\m{L}}$-stable subset. 
    Then $\pi_I$ is an $\ell$-pregroup homomorphism from $\m L$ to $\m{Aut}(\m{I})\wr \m{F}_n(\ZZ)$.
\end{lemma}
\begin{proof}
    It is clear that $\pi_I$ preserves the neutral element and the lattices connectives, since they are defined coordinate-wise.
    Also, it follows from the stability of $I$ that the map $\widetilde{f} \mapsto \widetilde{f}|_{I}$ is a group homomorphism.
    Therefore for each $f\in L$, we have 
    \begin{align*}
    \pi_I(f)^\ell
    &= (\widetilde{f}|_{I},(\overline{f}_j)_{j\in I})^\ell \\
    &= ((\widetilde{f}|_{I})^{-1}, (\overline{f}_{(\widetilde{f}|_{I})^{-1}(j)}^\ell)_{j\in I}) \\
    &= (\widetilde{f}^{-1}|_{I}, (\overline{f}_{\widetilde{f}^{-1}|_{I}(j)}^\ell)_{j\in I}) \\
    &= \pi_I(f^\ell).
    \end{align*}
    and similarly for $f^r$. Moreover, all $f,g\in L$, we have
    \begin{align*}
        \pi_I(f)\pi_I(g) 
        &= (\widetilde{f}|_{I},(\overline{f}_j)_{j\in I})(\widetilde{g}|_{I},(\overline{g}_j)_{j\in I}) \\
        &=(\widetilde{f}|_{I}\widetilde{g}|_{I},(\overline{f}_{\widetilde{g}|_{I}(j)}\overline{g}_j)_{j\in I}) \\
        &= ((\widetilde{f}\widetilde{g})|_{I}(\overline{f}_{\widetilde{g}|_{I}(j)}\overline{g}_j)_{j\in I}) \\
        &=\pi_I(fg).
    \end{align*}
    Hence  $\pi_I$ is an $\ell$-pregroup homomorphism.
\end{proof}

Similar to the case of $\ell$-groups we show that  subdirectly irreducible periodic $\ell$-pregroups have a transitive representation. A representation $\m{L}\leq \m{Aut}(\m{J})\wr \m{F}_n(\ZZ)$ is called \emph{transitive} if $\glob{\m{L}}$ acts transitively on $J$. 

\begin{lemma}\label{l: transitive}
    Let $\m{L}$ be a  subdirectly irreducible $n$-periodic $\ell$-pregroup. 
    Then $\m{L}$ has a transitive representation $\m{L}\leq \m{Aut}(\m{J})\wr \m{F}_n(\ZZ)$.
\end{lemma}
\begin{proof}
    Consider any representation $\m{L}\leq \m{Aut}(\m{J})\wr \m{F}_n(\ZZ)$. 
    For each $j\in J$, let $I_j$ be the orbit of $j$ under the action of $\glob{\m{L}}$, so $\glob{\m{L}}$ acts transitively on $I_j$. 
    By Lemma~\ref{lemma:pi_I-hom}, for each $j\in J$, $\pi_{j} := \pi_{I_j}$ is a homomorphism from $\m L$ to $\m{Aut}(\m{I}_j)\wr \m{F}_n(\ZZ)$ and, since  $J = \bigcup_{j\in J} I_j$, it follows that $\bigcap_{j\in J} \ker(\pi_{j}) = \Delta_L$. 
    Therefore, since $\m{L}$ is  subdirectly irreducible, there exists a $j \in J$ with $\ker(\pi_j) = \Delta_L$, i.e., $\pi_j$ is injective, so $\m L$ embeds  into $\m{Aut}(\m{I}_j)\wr \m{F}_n(\ZZ)$ via $\pi_j$.
    This yields the desired transitive representation $\m{L}\leq\m{Aut}(\m{I}_j)\wr \m{F}_n(\ZZ)$.
\end{proof}

\begin{remark}
    If $\m{L}$ is an $n$-periodic proper $\ell$-pregroup, i.e., not an $\ell$-group, and $\m{L}\leq \m{Aut}(\m{J})\wr \m{F}_n(\ZZ)$ is a transitive representation, then there exists $i \in \loc{\supp}(\m{L})$, yielding a local element $f\in L$ with $\overline{f}_i \neq 1$. Now if $j\in J$, then, by transitivity, there exists a $g\in \grp{L}$ with $\widetilde{g}(j)= i$. Hence for the local element $h=g^{-1}fg$, we get  $\overline{h}_j = \overline{(g^{-1}fg)}_j = \overline{g}_j^{-1}\overline{f}_{\widetilde{g}(j)}\overline{g}_j \neq 1$ and thus $j\in \loc{\supp}(\m{L})$. We conclude that  $J = \loc{\supp}(\m{L})$. 
\end{remark} 

The next lemma shows that for transitive representations the local subalgebra is a subdirect power with respect to the local component of the wreath product.

\begin{lemma}\label{lemma:FSI-subdirect-local}
    Let $\m{L}$ be a  properly $n$-periodic $\ell$-pregroup with transitive representation  $\m{L}\leq \m{Aut}(\m{J})\wr \m{F}_n(\ZZ)$. 
    Then the  restrictriction $\loc{\m{L}}\leq \m{F}_n(\ZZ)^J$ of the representation is subdirect, i.e., for each $i\in J$, the projection $\pi_i\colon \loc{\m{L}} \to \m{F}_n(\ZZ)$ to the $i$th coordinate is surjective.
\end{lemma}

\begin{proof}
    Corollary~3.20 of \cite{GS} states, specializing to our case, that if $a,b \in \mathrm{F}_n(\ZZ)$ are elements with periodicities $r$ and $s$, respectively, then they generate an element of periodicity $\lcm\{r,s\}$. So, it is enough to show that for each $j\in J$, prime number $p$ and $k\in \ZZ^+$ with $p^k \mid n$, there exists an $f\in \loc{L}$ such that $\overline{f}_j$ has periodicity divisible by $p^{k}$. 
    Moreover, since the representation is transitive, if $\overline{f}_{i}$ has periodicity divisible by $p^k$, then for $g\in \Grp{L}$ with $\widetilde{g}(j) = i$, we have that $\overline{(g^{-1}fg)}_j = \overline{g}_j^{-1}\overline{f}_{\widetilde{g}(j)} \overline{g}_j = \overline{g}_j^{-1}\overline{f}_{i} \overline{g}_j $ has periodicity divisible by $p^k$. 
    Therefore it is enough to find $i\in J$ and $f \in \loc{L}$ such that $\overline{f}_i$ has periodicity divisible by $p^k$. 
    
    Since $\m{L}$ is properly $n$-periodic, by Proposition~\ref{prop:period-subalg}, there exists a positive idempotent $f\in \loc{L}$ with periodicity divisible by $p^k$. 
    Moreover, Lemma~5.5 of \cite{GS} states that $\per(f) = \lcm\{\per(\overline{f}_i) : i\in J\}$, so there is an $i\in J$ such that $\overline{f}_i$ has periodicity divisible by $p^k$. 
\end{proof}

\begin{corollary}\label{cor:proper-per-FSI-Fn}
    For every  subdirectly irreducible properly $n$-periodic $\ell$-pregroup  $\m{L}$, $\m{F}_n(\ZZ) \in \vr(\m{L})$.
\end{corollary}

\begin{proof}
By Lemma~\ref{l: transitive}, $\m L$ has a transitive representation so, by Lemma~\ref{lemma:FSI-subdirect-local}, $\m{F}_n(\ZZ)$ is a homomorphic image of the subalgebra $\loc{\m{L}}$ of $\m L$. 
\end{proof}

Now we are able to prove the axiomatization result.

\begin{proposition}
    For all $n\in \ZZ^+$, $\bigvee_{k= 1}^n \mathsf{LP}_k$ is axiomatized by $\g_n(x)\g_n(x)^{\ell} \approx 1$.
\end{proposition}

\begin{proof}   
      $\bigvee_{k= 1}^n \mathsf{LP}_k$ satisfies the equation  $\g_n(x)\g_n(x)^{\ell} \approx 1$, as observed in \cite[Remark~2.5]{GS}. 
      Conversely, let $\m{L} \in \va{DLP}$ be a  subdirectly irreducible that satisfies the equation. 
      If  $\m{L}$ were properly $m$-periodic for some $m>n$, then by Corollary~\ref{cor:proper-per-FSI-Fn} we would get $\m{F}_m(\ZZ) \in \vr(\m{L})$, contradicting the fact that $\m{F}_{m}(\ZZ)$ fails the equation $\g_n(x)\g_n(x)^{\ell} \approx 1$ for all $m>n$. 
      So, $\m{L}\in \bigvee_{k= 1}^n \mathsf{LP}_k$.
\end{proof}

\section*{Acknowledgments}
This project has received funding from the European Union’s Horizon 2020 research and innovation programme under the Marie Skłodowska-Curie grant agreement No 101007627. The second author was supported by the Swiss National Science Foundation (SNSF), grant no. 200021\textunderscore215157.
No use of Artificial Intelligence was made in obtaining or presenting the results in this manuscript. (The results were presented in the conferences BLAST in May 2026 and TACL in July 2026.)

\end{document}